\documentclass[11pt ,twoside]{amsart}  
\usepackage[left=2.5cm,right=2.5cm,top=2.5cm,bottom=2.5cm]{geometry}
\usepackage{graphicx}
\usepackage{multirow}
\usepackage{hyperref}
\usepackage{bbm}
\usepackage{amsmath,amssymb,amsfonts, amsthm,epsfig}
\usepackage{verbatim,color}
\usepackage{enumerate}
\usepackage{mathtools}
\mathtoolsset{showonlyrefs}

\numberwithin{equation}{section}

\newtheorem{thma}{Theorem}[section]
\newtheorem{lemma}[thma]{Lemma}

\newtheorem{prop}[thma]{Proposition}

\newtheorem{claim}{Claim}
\newtheorem{remark}{Remark}

\newcommand{\eps}{\varepsilon}

\renewcommand{\mod}{\;\text{mod}\;}

\begin{document}

\title{Airy point process at the liquid-gas boundary}

\author{Vincent Beffara, Sunil Chhita and Kurt Johansson} 

\thanks{V.B. and S.C. gratefully acknowledge the support of the German
  Research  Foundation in  SFB 1060-B04  ``The Mathematics  of Emergent
  Effects''. K.J.  gratefully acknowledge the  support of the  Knut and
  Alice Wallenberg Foundation grant KAW:2010.0063. }

\maketitle
\begin{abstract}

Domino tilings of the two-periodic Aztec diamond feature all of the three possible types of phases of random tiling models. These phases are determined by the decay of correlations between dominoes and are generally known as solid, liquid and gas. The liquid-solid boundary is easy to define microscopically and is known in many models to be described by the Airy process in the limit of a large random tiling. The liquid-gas boundary has no obvious microscopic description. Using the height function we define a random measure in the two-periodic Aztec diamond designed to detect the long range correlations visible at the liquid-gas boundary. We prove that this random measure converges to the extended Airy point process. This indicates that, in a sense, the liquid-gas boundary should also be described by the Airy process.

\end{abstract}
\tableofcontents

\section{Introduction} \label{sec:intro}

\subsection{The two-periodic Aztec diamond and random tilings} \label{sec:intro:overview}

An \emph{Aztec diamond graph of size $n$} is a bipartite graph which contains white vertices given by
\begin{equation}
\mathtt{W}= \{(i,j): i \mod 2=1, j \mod 2=0, 1 \leq i \leq 2n-1, 0 \leq j \leq 2n\}
\end{equation}
and black vertices given by
\begin{equation}
\mathtt{B}= \{(i,j): i \mod 2=0, j \mod 2=1, 0 \leq i \leq 2n, 1\leq j \leq 2n-1\}.
\end{equation}
The edges of the Aztec diamond graph are given by $\mathtt{b} -\mathtt{w}= \pm e_1,\pm e_2$ for $\mathtt{b} \in \mathtt{B}$ and $\mathtt{w} \in \mathtt{W}$, where $e_1=(1,1)$ and $e_2=(-1,1)$.  The coordinate of a face in the graph is defined to be the coordinate of its center.
For an Aztec diamond  graph of size $n=4m$ with $m \in \mathbb{N}$, define the \emph{two-periodic Aztec diamond} to be an Aztec diamond graph with edge  weights $a$ for all edges incident to the faces $(i,j)$ with  $(i+j)\mod 4=2$ and edge weights $b$ for all the edges incident to the faces $(i,j)$ with  $(i+j) \mod 4=0$; see the left figure in Fig.~\ref{fig:weights}. We call the faces $(i,j)$ with $(i+j)\mod 4=2$ to be the \emph{$a$-faces}. 
  For the purpose of  this paper, we set $b=1$; this  incurs no loss of
  generality, since multiplying both $a$  and $b$ by the same constant
  does not change the model that we consider.

\begin{center}
\begin{figure}
\includegraphics[height=3in]{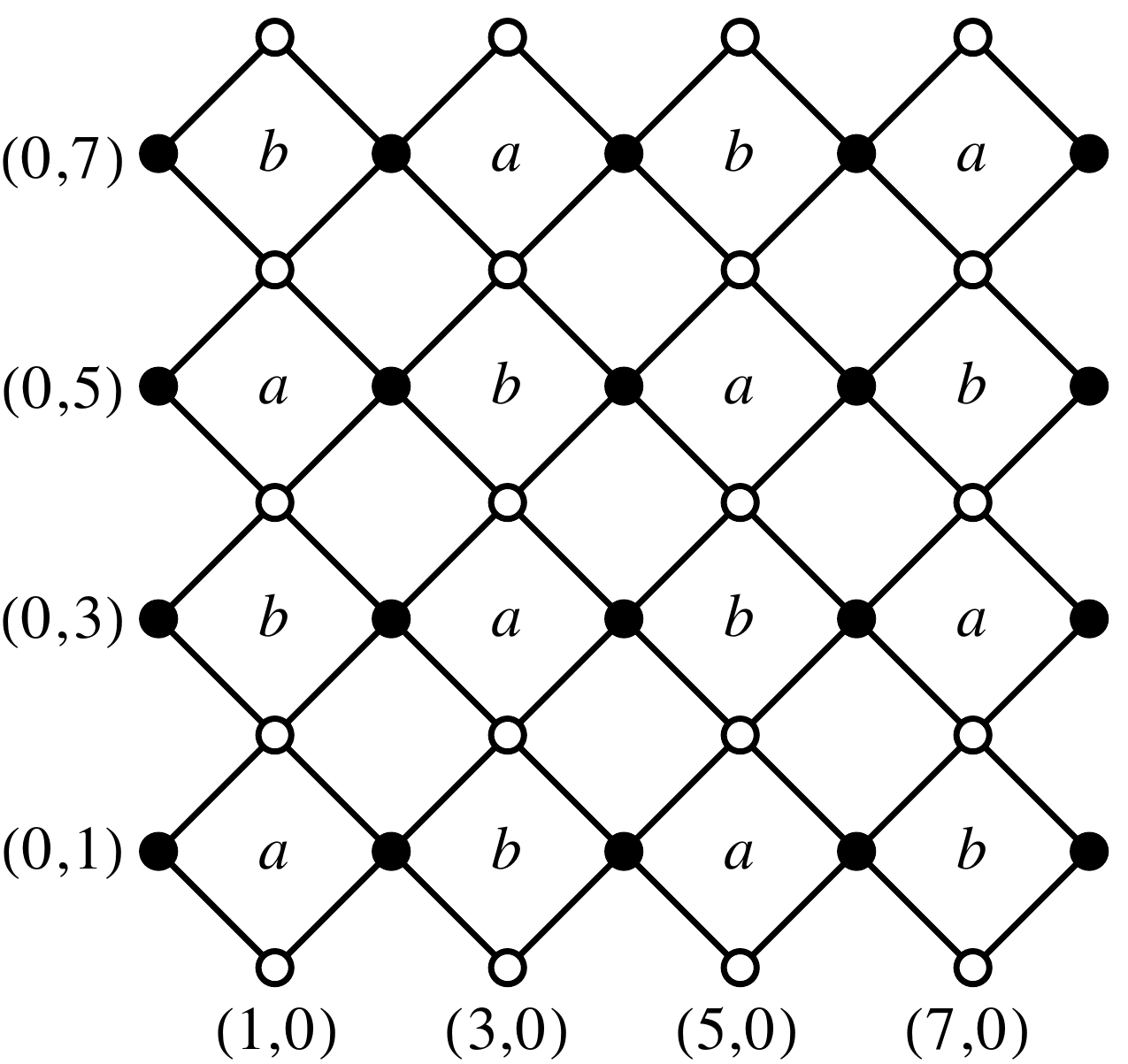}
\hspace{2mm}
\includegraphics[height=3in]{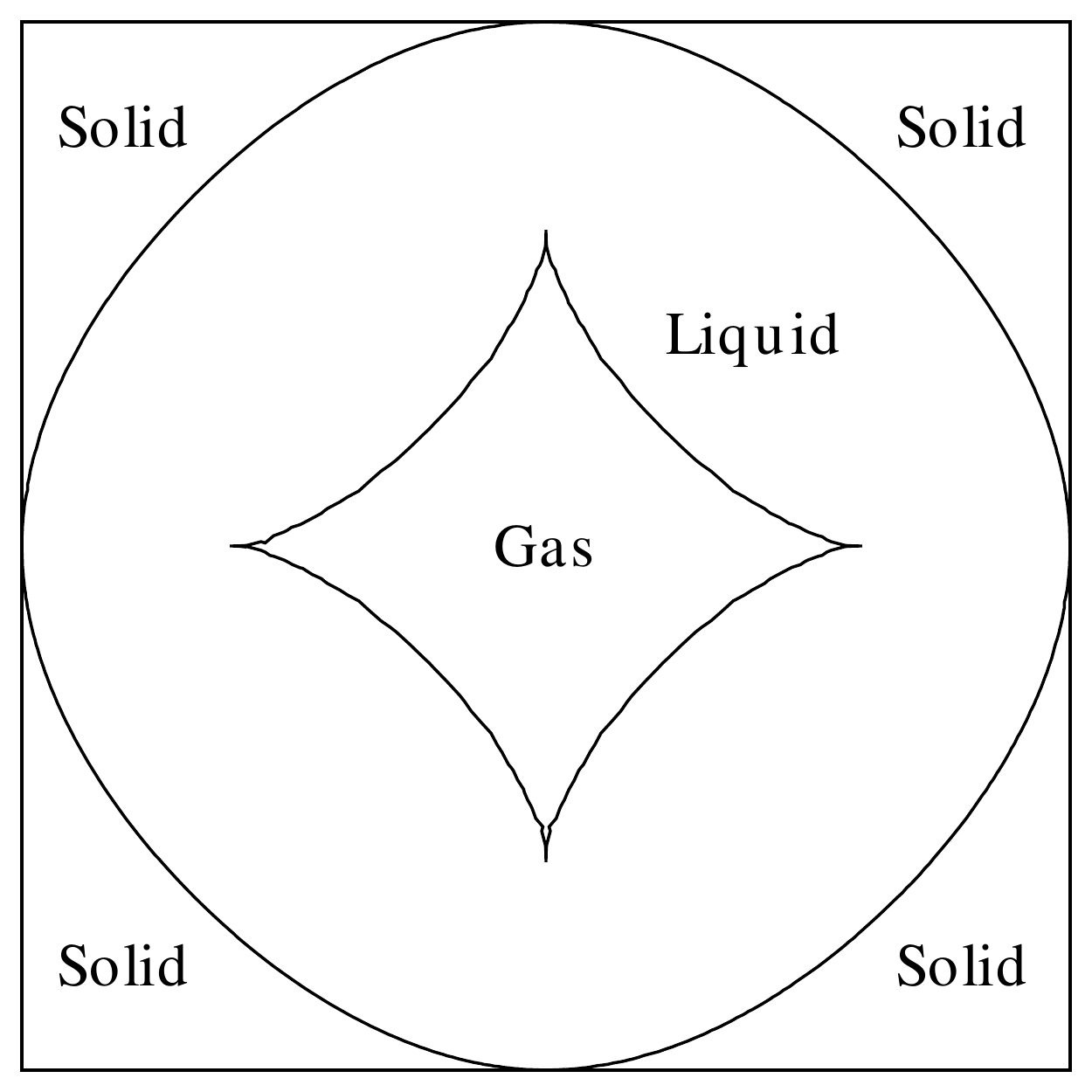}
\caption{The left figure shows the two-periodic Aztec diamond graph for $n=4$ with the edges weights  given by $a$ (or $b$) if the edge is incident to a face marked $a$ (respectively $b$). The right figure shows the limit shape when $a=0.5$ and $b=1$. See~\cite{KO:07, FSG:14, CJ:16} for the explicit curve.}
\label{fig:weights}
\end{figure} 
\end{center}

A \emph{dimer} is an edge and a \emph{dimer covering} is a subset of edges so that each vertex is incident to exactly one edge. Each dimer covering of the two-periodic Aztec diamond is picked with probability proportional to the product of  its edge weights. For the two-periodic Aztec diamond, it is immediate that each dimer covering is picked uniformly at random when $a=1$.  

An equivalent notion to a dimer covering is a domino tiling, where one replaces each dimer by a domino. Each dimer is  the graph theoretic dual of a domino. Simulations of domino tilings of large bounded regions exhibit interesting features due to the emergence of a  \emph{limit shape}. Here, the tiling \emph{separates} into distinct macroscopic regions: \emph{solid}, where the configurations of dominoes are deterministic; \emph{liquid}, where the correlation between dominoes have polynomial decay in distance; and \emph{gas}  where the dominoes have exponential decay of correlations. 
These phases are characterized in~\cite{KOS:06} but first noticed in~\cite{NHB:84}.
Even though their names may suggest otherwise, these regions are not associated with physical states of matter.   An alternate convention is to say that the solid region is the \emph{frozen} region while the liquid and gas regions are the \emph{unfrozen} regions.  The liquid region is then referred to as  the \emph{rough unfrozen} region (or simply rough region) and the gas region is referred to as the  \emph{smooth unfrozen} region.


Considerable research attention has been directed to tiling models, including domino tilings and \emph{lozenge tilings}, on bounded regions whose limit shapes contain both solid and liquid regions, but no gaseous regions.  The primary reason behind this attention lies in the fact that in several cases these models are mathematically tractable due to direct connections  with algebraic combinatorics through the so-called \emph{Schur processes}~\cite{OR:03}.  
 By exploiting the algebraic structure via the so-called Schur generating functions, it is possible to find the limit shape, and the local and global bulk limiting behaviors in several cases; see the recent articles~\cite{BK:16,  Gor:16, Pan:15}. 
 More computational approaches are also used to find these asymptotic quantities as well as  the limiting edge behavior. 
 These approaches often use in an essential way that the dimers, or some associated particles, form a determinantal process with an explicit correlation kernel. 
Finding this correlation kernel is not a simple task in general, but successful techniques  have come from applying the Karlin-McGregor Lindstr\"{o}m Gessel Viennot matrix and the Eynard Mehta theorem, or using vertex operators; see~\cite{BR:05} for the former and~\cite{BCC:14, BBCCR:15} for a recent exposition of the latter. The limiting correlation kernels that appear are often the same, or related to, those that occur in random matrix theory. Indeed, in~\cite{FS:03, Joh:05,OR:03, Pet:14}, the limiting behaviour of the random curve separating the solid and liquid regions is described by the \emph{Airy process}, a universal probability distribution first appearing in~\cite{PS:02} in connection with random growth models.

Domino tilings of the two-periodic Aztec diamond do not belong to the Schur process class. In fact the two techniques mentioned above fail (at least for us) for this model.  However, in~\cite{CY:13}, the authors derive a formula for the correlations of dominoes for the two-periodic Aztec diamond, that is they give a formula for the so-called \emph{the inverse Kasteleyn matrix}; see below for more details. The formula given in~\cite{CY:13} is particularly long and without any specific algebraic or asymptotic structure. In~\cite{CJ:16}, the formula is dramatically simplified and written in a good form for asymptotic analysis. Precise asymptotic expansions of the inverse Kasteleyn matrix reveal the limit shape as well as the asymptotic entries of the inverse Kasteleyn matrix in all three macroscopic regions, and at the solid-liquid and liquid-gas boundaries.  Due to technical considerations, these asymptotic computations were only performed along the main diagonal of the two-periodic Aztec diamond.  Roughly speaking, the outcome is that the asymptotics of the inverse Kasteleyn matrix at the liquid-gas boundary is given by a mixture of a dominant `gas part' and a lower order `Airy part'  correction.
 Unfortunately, these asymptotic results only describe the statistical behavior of the dominoes at the liquid-gas boundary, and do not determine the \emph{nature} of this boundary.  More explicitly, it is highly plausible, as can be seen in simulations, see Fig.~\ref{fig:tpn200}, that there is a family of lattice paths which separate the liquid and gas regions. The exact microscopic definition of these paths is not clear, see~\cite{CJ:16} for a discussion and a suggestion. The asymptotic computations in~\cite{CJ:16} do not give us any information about these paths. At the liquid-solid boundary the definition of the boundary is obvious, it is the first place where we see a deviation from the regular brick-wall pattern.
At the liquid-gas boundary however, these paths, if they exist, are in some sense `sitting' in a `gas' background. The paths represent long-distance correlations and the purpose of this paper is to extract these correlations from the background `gas noise' and show that they are described by the Airy point process. This strongly indicates that there should be a random boundary path at the liquid-gas boundary which, appropriately rescaled, converges to the Airy process just as at the liquid-solid boundary. 
 
We approach the problem via the so called \emph{height function} of the domino tiling, an idea originally introduced by Thurston~\cite{Thu:90}. The height function is defined for the two-periodic Aztec diamond at the center of each face of the Aztec diamond graph, characterized by the height differences in the following way: 
if there is a dimer covering the edge shared between two faces, the height difference between the two faces is $\pm 3$, while if there is no dimer covering the shared edge between two faces the height difference is $\mp 1$.  We use the convention that as we traverse from one face to an adjacent face, the height difference will be $+3$ if there is a dimer covering the shared edge and the left vertex of the incident edge is black.   Similarly the height difference is $+1$ when we cross an empty edge with a white vertex to the left. We assign the height at the face $(0,0)$ (outside of the Aztec diamond graph) to be equal to 0.  The height function on the faces bordering the Aztec diamond graph are deterministic and given by the above rule.  Fig.~\ref{fig:heights} shows a domino tiling of the Aztec diamond with the heights labeled at each face.
\begin{figure}
\begin{center}
\includegraphics[height=3in]{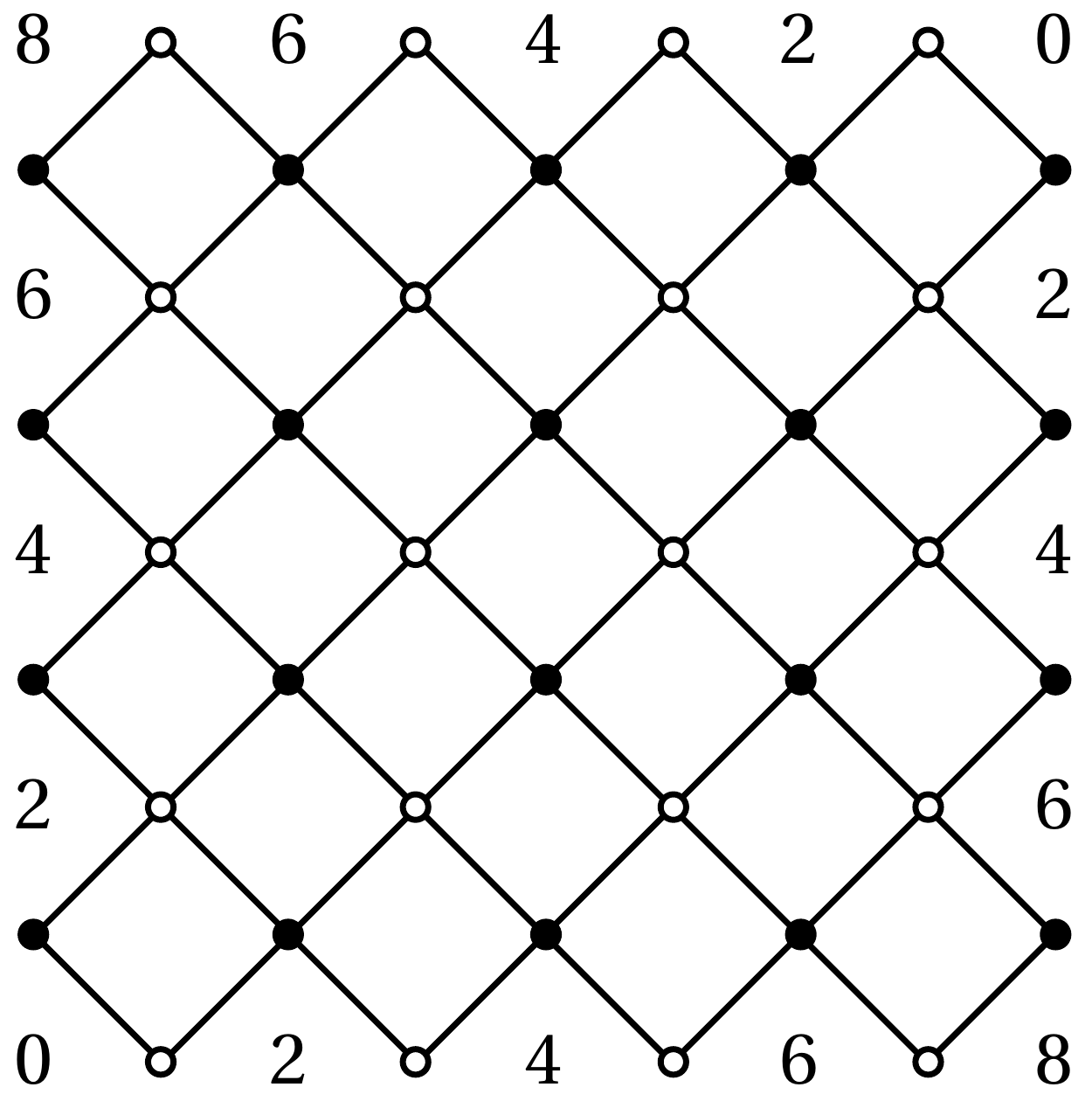}
\hspace{5mm}
\includegraphics[height=3in]{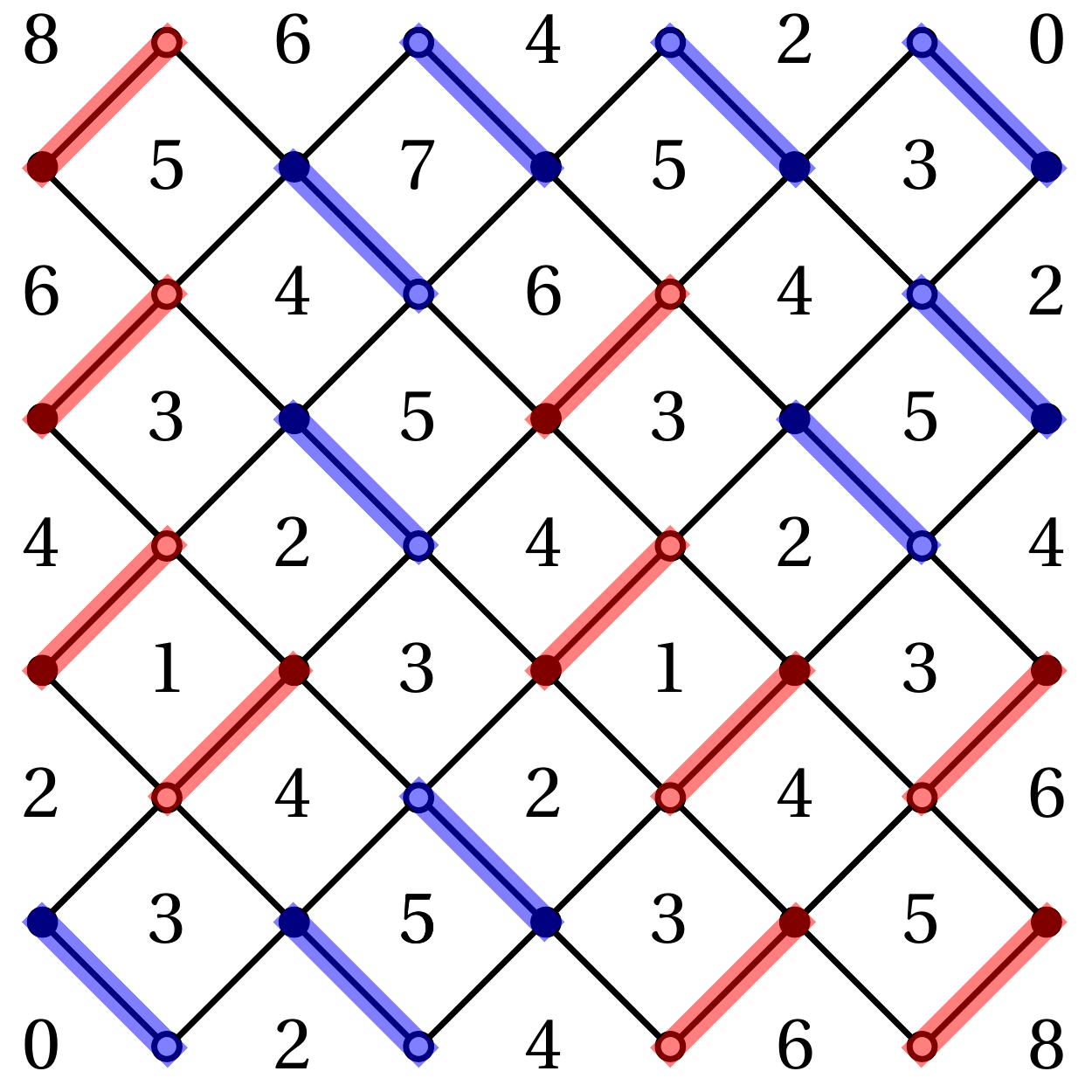}
\caption{The left figure shows the height function of an Aztec diamond graph of size 4 on the bordering faces imposing that the height at face $(0,0)$ is $0$. The right figure shows the same graph with a dimer covering and its corresponding height function.}
\label{fig:heights}
\end{center}
\end{figure}

\begin{figure}
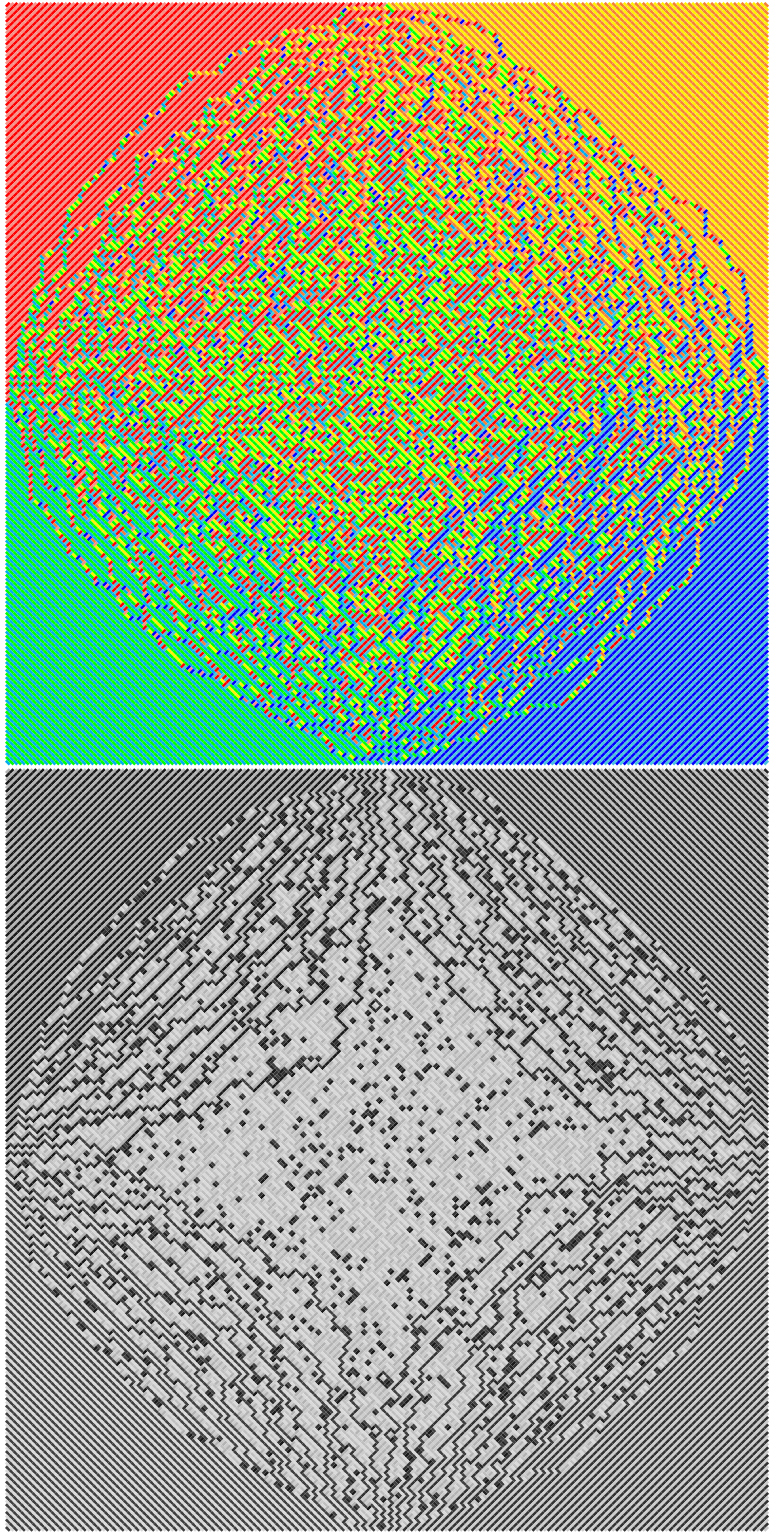

\begin{center}
\includegraphics[height=4in]{ttpn200a05.pdf}
\includegraphics[height=4in]{ttpn200a05bw.pdf}
\caption{Two different drawings of a simulation of a domino tiling of the two-periodic Aztec diamond of size 200 with $a=0.5$ and $b=1$.  The top figure contains eight different colors, highlighting the solid and liquid phases. The bottom figure contains eight different gray-scale colors to accentuate the gas phase. We choose $a=0.5$ for aesthetic reasons in relation of the size of the Aztec diamond and the size of the liquid and gas regions. }
\label{fig:tpn200}
\end{center}
\end{figure}

The height function has a limit shape that is the solution of a certain variational problem and this also, in principle, leads to a description of the macroscopic boundaries between the regions with different phases,~\cite{CKP:01,KO:07}. For the two-periodic Aztec diamond this gives the algebraic equation for the boundaries seen in Fig.~\ref{fig:weights}. The solid and gas phases correspond to flat pieces, \emph{facets}, of the limit shape. The solid phase has a completely flat height function even at the microscopic level, whereas the height function in the gas phase has small fluctuations.  
The global fluctuations of the height function in the liquid, or  rough
phase,  for  many  tilings  models have  been  studied  revealing  the
so-called   \emph{Gaussian  free   field}  in   the  limit;   see  the
papers~\cite{BF:08, BG:16,  Bou:07a, deTil:07, Dub:15,  DG:15, Dui:13,
  Dui:15, Ken:00, Ken:01, Pet:15} for examples with varying techniques
of proof. Hence, the liquid-gas interface can be seen as an example of
a boundary  between a rough  random crystal  surface and a  facet with
small, random, almost independent and Poissonnian dislocations. 

The novelty of this paper is that we use the height function close to the liquid gas boundary to introduce a random measure, defined in detail below, which captures the long distance changes in the height function, but \emph{averages out} the local height fluctuations coming from the surrounding gas phase.  This random measure gives a partial explanation of the nature of the liquid-gas boundary.

\subsection{Definition of the random measure}

Let $\mathbbm{I}_{A}$ be the indicator function for some set $A$ and denote $\mathbbm{I}$ to be the identity matrix or operator. 
Let $\mathrm{Ai}(\cdot)$ denote the standard Airy function, and define
\begin{equation} \label{eq:Airymod}
\begin{split}
\tilde{\mathcal{A}}(\tau_1,\zeta_1;\tau_2,\zeta_2)&= 
\int_0^\infty e^{-\lambda (\tau_1-\tau_2) } \mathrm{Ai} (\zeta_1 +\lambda) \mathrm{Ai} (\zeta_2+\lambda) d\lambda\\
\end{split}
\end{equation}
and
\begin{equation}\label{eq:Airyphi}
\phi_{\tau_1,\tau_2} (\zeta_1 ,\zeta_2) =
 \mathbbm{I}_{\tau_1<\tau_2} \frac{1}{\sqrt{4 \pi (\tau_2-\tau_1)}} e^{-\frac{(\zeta_1-\zeta_2)^2}{4(\tau_2-\tau_1)}-\frac{(\tau_2-\tau_1)(\zeta_1+\zeta_2)}{2}+\frac{(\tau_2-\tau_1)^3}{12}},
\end{equation}
the latter is referred to as the \emph{Gaussian part} of the extended Airy kernel; see~\cite{Joh:03}. The \emph{extended Airy kernel}, ${\mathcal{A}}(\tau_1,\zeta_1;\tau_2,\zeta_2)$, is defined by
\begin{equation}\label{eq:extendedAiry}
{\mathcal{A}}(\tau_1,\zeta_1;\tau_2,\zeta_2)=\tilde{\mathcal{A}}(\tau_1,\zeta_1;\tau_2,\zeta_2)-\phi_{\tau_1,\tau_2} (\zeta_1 ,\zeta_2).
\end{equation}

Let $\beta_{1} <\dots<\beta_{L_1}$, $L_1 \geq 1$, be given fixed real numbers.  The extended Airy kernel gives a determinantal point process on $L_1$ lines $\{\beta_1,\dots,\beta_{L_1}\} \times \mathbb{R}$.  We think of this process as a random measure $\mu_{\mathrm{Ai}}$ on $\{\beta_1,\dots,\beta_{L_1}\} \times \mathbb{R}$ in the following way:

Let $A_1,\dots , A_{L_2}$, $L_2 \geq 1$, be finite, disjoint intervals in $\mathbb{R}$ and write 
\begin{equation} \label{phi}
\Psi(x)= \sum_{p=1}^{L_2} \sum_{q=1}^{L_1} w_{p,q} \mathbbm{I}_{\{\beta_q\} \times A_p}(x)
\end{equation}
for $x \in \{\beta_1,\dots,\beta_{L_1}\} \times \mathbb{R}$, where $w_{p,q}$ are given complex numbers.  Then,
\begin{equation}
\mathbb{E} \left[ \exp \left( \sum_{p=1}^{L_2} \sum_{q=1}^{L_1} w_{p,q} \mu_{\mathrm{Ai}} (\{\beta_q\} \times A_p) \right) \right]
	= \det \left(  \mathbbm{I} + ( e^{\Psi}-1) \mathcal{A} \right)_{L^2(\{\beta_1,\dots,\beta_{L_1}\} \times \mathbb{R})}
\end{equation}
for $w_{p,q} \in \mathbb{C}$, defines the random measure $\mu_{\mathrm{Ai}}$, the \emph{extended Airy point process}.  

The positions of the particles in  the extended Airy point process can
be    thought    of    as    the   intersections    of    the    lines
$\{\beta_1,\dots,\beta_{L_1}\}  \times \mathbb{R}$  with  a family  of
random curves (a \emph{line ensemble};  see \cite{CH:14}). If we think
of  these  lines  as  level   lines  of  some  height  function,  then
$\mu_{\mathrm{Ai}} (\{\beta_q\} \times A)$ is the height change in $A$
along the  line $\{\beta_q\}  \times A$.  We want  to define  a random
measure in a random tiling of  the two-periodic Aztec diamond close to
the  liquid-gas  boundary  which  captures the  long  distance  height
differences, and  which converges  to $\mu_{\mathrm{Ai}}$.  Take $L_1$
lines in the  Aztec diamond at distances of order  $m^{2/3}$ from each
other,  and  look at  the  height  differences  along these  lines  in
intervals of length $O(m^{1/3})$ close  to the liquid-gas boundary. In
Fig.~\ref{fig:tpn200}, we  see something like long  random curves, but
the height differences along an interval  in the diamond will come not
only from these  curves but also from the smaller  sized objects which
are in  a sense  due to  the gas-like features  in the  background. We
expect  that  these  smaller  sized  objects  are  much  smaller  than
$O(m^{1/3})$,  and we  further  expect that  the correlations  between
these smaller sized  objects decay rapidly. We isolate  the effects of
the long  curves by  taking averages of  the height  differences along
copies of  the intervals on  $M$ parallel  lines, with $M$  tending to
infinity slowly  as $m$ tends  to infinity. The distances  between the
copies are  of order $(\log m)^2$,  which is less than  $m^{2/3}$, but
large enough  for the short range  correlations to decay. We  will now
make these  ideas precise and  define a  random measure on  $\mu_m$ on
$\{\beta_1,\dots,\beta_{L_1}\}  \times \mathbb{R}$  that we  will show
converges to $\mu_{\mathrm{Ai}}$.

The following constants come from the asymptotic results for the inverse Kasteleyn matrix for the two-periodic Aztec diamond; see~\cite{CJ:16} and Theorem~\ref{Airyasymptotics} below. Let
\begin{equation} \label{eq:parameterc}
c=\frac{a}{(1+a^2) }
\end{equation}
 which occurs throughout the paper. For this paper, we fix
 $\xi=-\frac{1}{2}\sqrt{1-2c}$ and set
\begin{equation} \label{eq:scalingparameters}
c_0=\frac{(1-2c)^{\frac{2}{3}}}{(2c(1+2c))^{\frac{1}{3}}}, \hspace{5mm} 
 \lambda_1 =  \frac{\sqrt{1-2c}}{2c_0} \hspace{5mm} \mbox{and} \hspace{5mm}
\lambda_2=\frac{(1-2c)^{\frac{3}{2}}}{2c c_0^2}.
\end{equation}
The term $\xi$ can be thought of as the asymptotic parameter which puts the analysis at the liquid-gas boundary after re-scaling (along the main diagonal in the third quadrant of the rotated Aztec diamond). The terms $\lambda_1$ and $\lambda_2$ are scale parameters, as found in~\cite{CJ:16}.

We will define discrete lines $\mathcal{L}_m(q,k)$, $q \in \{1,\dots,L_1\}$, $1\leq k \leq M$, which we should think of as $M$ copies of the lines $\{\beta_1,\dots,\beta_{L_1}\} \times \mathbb{R}$ embedded in the Aztec diamond as mentioned above.  Recall that $e_1=(1,1)$ and $e_2=(-1,1)$.
Set
\begin{equation}
\beta_m(q,k)=2[ \beta_q \lambda_2 (2m)^{2/3} + k \lambda_2 (\log m )^2 ] 
\end{equation}
and define 
\begin{equation}
\mathcal{L}_m(q,k)= \mathcal{L}_m^0(q,k) \cup \mathcal{L}_m^1(q,k),
\end{equation}
where, for $\eps \in \{0,1\}$,
\begin{equation}
	\mathcal{L}_m^{\eps}(q,k) = \{(2t-\eps+\frac{1}{2} ) e_1 -\beta_m(q,k)e_2 ; t \in [0,4m]\cap \mathbb{Z} \}.
\end{equation}
The lines $\mathcal{L}_m(q,k)$ are \emph{discrete lines} parallel to the main diagonal with direction $(1,1)$; see Fig.~\ref{fig:Lmqk}.
\begin{figure}
\begin{center}
\includegraphics[height=5cm]{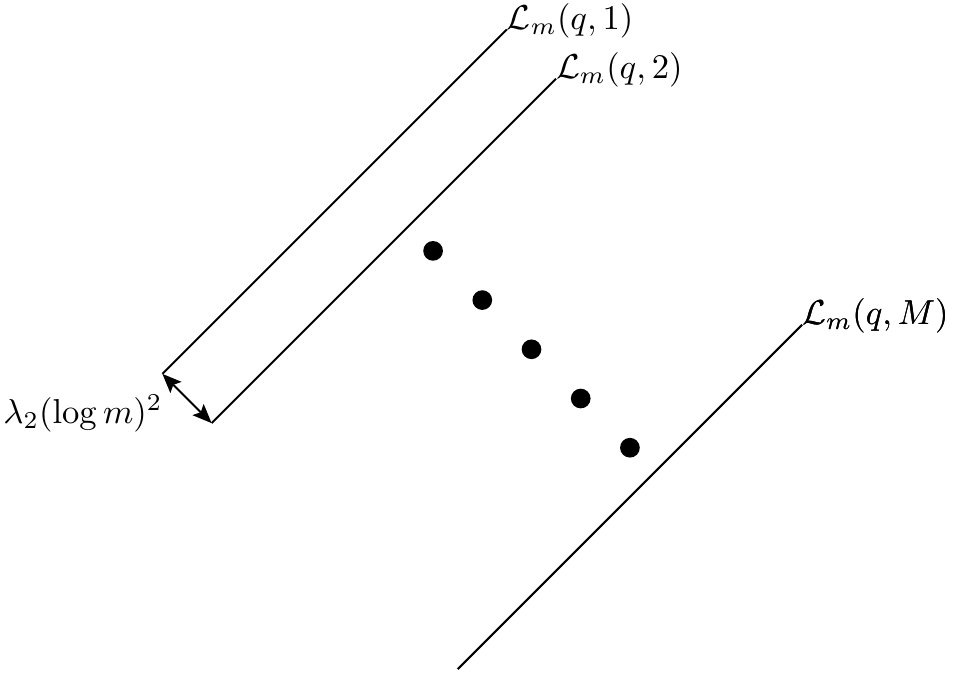}
\caption{The lines $\mathcal{L}_m(q,1)$ to $\mathcal{L}_m(q,M)$}
\label{fig:Lmqk}
\end{center}
\end{figure}
Write 
\begin{equation}
\mathcal{L}_m^{\eps}= \bigcup_{q=1}^{L_1} \bigcup_{k=1}^M \mathcal{L}_m^{\eps}(q,k)
\end{equation}
and 
\begin{equation} \label{Lm}
\mathcal{L}_m= \mathcal{L}_m^{0} \cup \mathcal{L}_m^{1},
\end{equation}
so that $\mathcal{L}_m$ is the union of all these discrete lines. For $z \in \mathcal{L}_m$, we write
\begin{equation}
\eps (z)=\eps \hspace{5mm} \mbox{if } z \in \mathcal{L}_m^{\eps}.
\end{equation}
Each of the points in $\mathcal{L}_m^\eps$ may be covered by a dimer. When computing height differences, the sign of the height change as we cross a dimer depends on whether $\eps=0$ or $1$. Later, we will think of these dimers as particles and $\eps$ will then be called the \emph{parity} of the particle. We think of $\eps(z)=0$ having even parity while $\eps(z)=1$ having odd parity.

We call a subset $I \subseteq \mathcal{L}_{m}(q,k)$ a \emph{discrete interval} if it has the form 
\begin{equation}\label{discreteI}
I= \{(\frac{1}{2} +t) e_1 -\beta_m(q,k)e_2 ; t_1 \leq t < t_2  \}, 
\end{equation}
where $t_1,t_2 \in 2 \mathbb{Z}+1$. We denote the height of the face $F$ by $h(F)$ as defined in Section~\ref{sec:intro:overview}.  The $a$-faces adjacent to the discrete interval $I$ in~\eqref{discreteI} are defined to be the faces
\begin{equation}
F_+(I)=t_2e_1 -\beta_m(q,k) e_2
\end{equation}
\begin{equation}
F_-(I)=t_1e_1 -\beta_m(q,k) e_2,
\end{equation}
which are the \emph{end faces} of a discrete interval; see Fig.~\ref{fig:endpoints}.
\begin{figure}
\begin{center}
\includegraphics[height=6cm]{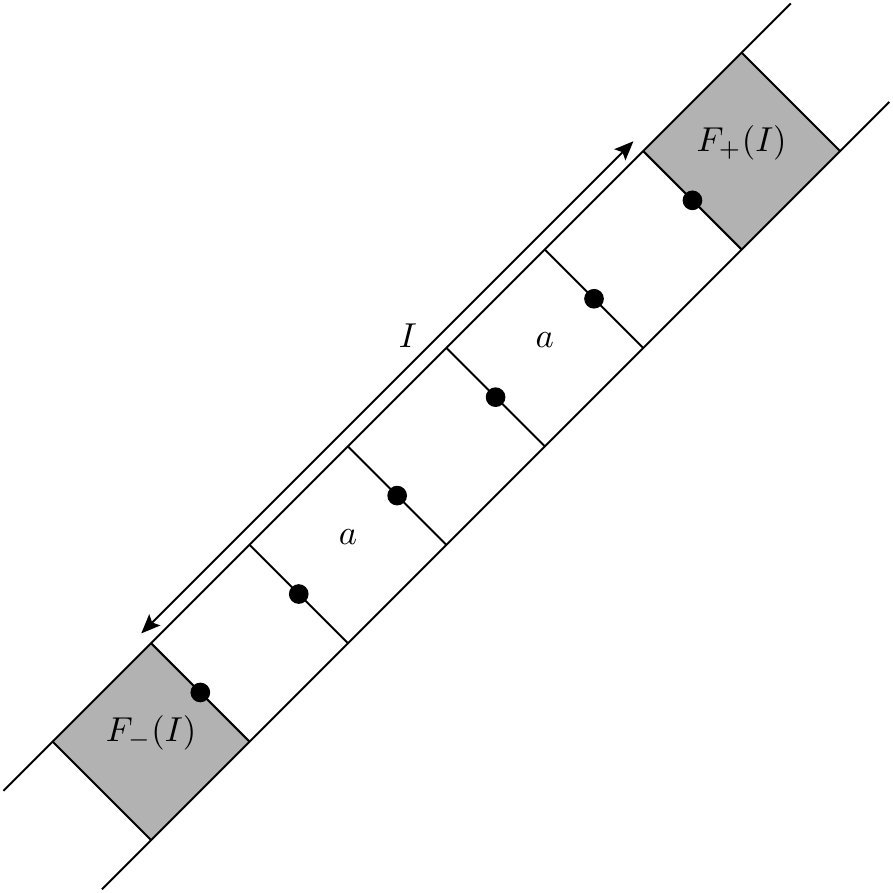}
\caption{The two endpoints of I: $F_-(I)$ and $F_+(I)$}
\label{fig:endpoints}
\end{center}
\end{figure}

The \emph{height difference} along $I$ is then 
\begin{equation}
\Delta h(I) = h(F_+(I))-h(F_-(I)).
\end{equation}

Write
\begin{equation}
\rho_m=4 [m(1+\xi)], \hspace{5mm} \tau_m(q)=[\beta_q^2 \lambda_1 (2m)^{1/3}].
\end{equation}
We want to embed the real line intervals $A_p=[\alpha_p^l,\alpha_p^r]$, $1 \leq p \leq L_2$, in the Aztec diamond as discrete intervals close to the liquid-gas boundary.  For this, and for the asymptotic analysis later in the paper, it is convenient to use the following parameterization of $\mathcal{L}_m(q,k)$.  Given $z \in \mathcal{L}_m(q,k)$, there is a $t(z) \in [-\rho_m/2+\tau_m(q), 4m-\rho_m/2+ \tau_m(q)] \cap \mathbb{Z}$ such that 
\begin{equation} \label{tz}
z=(\rho_m +2(t(z)-\tau_m(q)) -\eps(z) +\frac{1}{2} )e_1 -\beta_m(q,k)e_2.
\end{equation}
We also write, for $s \in \mathbb{Z}$, 
\begin{equation}\label{zqks}
z_{q,k}(s) = (\rho_m +s -2 \tau_m(q)+\frac{1}{2})e_1 - \beta_m (q,k) e_2,
\end{equation}
so that $\mathcal{L}_m(q,k)= \{z_{q,k} (s); s \in [-\rho_m+2\tau_m(q)-1, 8m-\rho_m+2 \tau_m(q)]\cap \mathbb{Z}\}$.

Let  
\begin{equation} \label{Apmtilde}
\tilde{A}_{p,m} = \{s \in \mathbb{Z} ; 2 [ \alpha_p^l \lambda_1 (2m)^{1/3}]-1 \leq s < 2 [ \alpha_p^r \lambda_1 (2m)^{1/3}] +1\}.
\end{equation}
The embedding of the interval $A_p$ as a discrete interval in $\mathcal{L}_{m}(q,k)$ is then given by
\begin{equation}
I_{p,q,k}=\{z_{q,k}(s); s \in \tilde{A}_{p,m} \}. \label{Ipqk}
\end{equation}
We define the random signed measure $\mu_m$ on $\{\beta_1,\dots,\beta_{L_1}\}\times \mathbb{R}$ by
\begin{equation}
\mu_m ( \{\beta_q\} \times A_p ) = \frac{1}{4 M } \sum_{k=1}^M \Delta h(I_{p,q,k}) \hspace{5mm} \mbox{for $1\leq p\leq L_2,1\leq q \leq L_1$}. \label{mu_m}
\end{equation}
The height changes between $a$-faces along a line are multiples of 4.
Intuitively, the factor $4$ in the above normalization ensures that we increase the height by $1$ for each connected component of $a$ edges traversing the two boundaries, which
in a sense are the paths describing the transition between the liquid and gas phases.

\subsection{Main Theorem}

We now state the main theorem of this paper.  Assume that $M=M(m)\to \infty$ as $m\to \infty$, but $M(m)( \log m)^2 /m^{1/3} \to 0$ as $m \to \infty$.  E.g.\ we could take 
$M=(\log m)^\gamma$ for some $\gamma>0$.

\begin{thma} \label{thm:moments}
The sequence of measures $\{\mu_m \}$ converges to $\mu_{\mathrm{Ai}}$ as $m$ tends to infinity in the sense that
that there is an $R>0$ so that for all $|w_{p,q}| \leq R$, $1 \leq p \leq L_2$, $1 \leq q \leq L_1$, 
\begin{equation} \label{Laplacelimit}
\lim_{m \to \infty} \mathbb{E} \left[ \exp \left( \sum_{p=1}^{L_2} \sum_{q=1}^{L_1} w_{p,q} \mu_{m} (\{\beta_q\} \times A_p) \right) \right]
=\mathbb{E} \left[ \exp \left( \sum_{p=1}^{L_2} \sum_{q=1}^{L_1} w_{p,q} \mu_{\mathrm{Ai}} (\{\beta_q\} \times A_p) \right) \right].
\end{equation}

\end{thma}

In the above equation, the expectation on the left side is with respect to the two-periodic Aztec diamond measure and the right side is with respect to the extended Airy point process.

\subsection{Heuristic Interpretation}

The asymptotic formulas at the liquid-gas boundary for \emph{the inverse Kasteleyn matrix}, described below, are computed in~\cite{CJ:16}.  These formulas have a primary contribution from the full-plane gas phase inverse Kasteleyn matrix and a correction term given in terms of the extended Airy kernel. This means that when we consider correlations between dominoes that are relatively close they are essentially the same as in a pure gas phase. However, at longer distances the correction term becomes important since correlations in a pure gas phase decay exponentially. A heuristic description of the behavior of the dominoes at the liquid-gas boundary is that the behavior is primarily a gas phase but there is a family of random curves which have a much longer interaction scale than the gas phase objects.  Although this is not quite an accurate description of the boundary, it naturally motivates the random measure $\mu_m$ defined in~\eqref{mu_m}.

In this paper, we do not investigate whether there is a natural geometric curve that separates the liquid and gas regions. A candidate for such a path, the last \emph{tree-path}, is discussed in~\cite[Section 6]{CJ:16}, but there are other possible definitions. Such a path should converge to the Airy$_2$ process. The present work can be thought of as a crucial  step in proving
this by providing  a specific \emph{averaging} of  the height function
at   the  liquid-gas   boundary  which   isolates  the   long-distance
correlations. However, it  does not give any  direct information about  the existence of a natural path that converges to the Airy$_2$ process.  We plan to investigate this in a future paper (work in progress).

\subsection{Organization}   
The rest of the paper is organized as follows: in Section~\ref{section:particles}, we introduce the particle description associated to the height function and the inverse Kasteleyn matrix.  In Section~\ref{section:asymptoticformulas}, we state asymptotic formulas and results needed for the rest of the paper.  The proof of Theorem~\ref{thm:moments} is given in Section~\ref{sec:mainthmproof}.  In Section~\ref{sec:U0proof}, we give the proof of lemmas that are used in the proof of Theorem~\ref{thm:moments}. Finally, in Section~\ref{sec:proofasymptotics}, we give the proof of the results stated in Section~\ref{section:asymptoticformulas}.

\subsection*{Acknowledgments}
All the authors wish to thank  the Galileo Galilei Institute for hospitality and support during the scientific program `Statistical Mechanics, Integrability and Combinatorics', which provided a useful platform for this work.   We would also like to thank Anton Bovier, Maurice Duits and Patrik Ferrari for useful discussions and the referees for useful comments and suggestions.  

\section{Inverse Kasteleyn matrix and the particle process} \label{section:particles}

In this section, we introduce a particle process which will be used to prove~\eqref{Laplacelimit}. This particle process enables the direct use of determinantal point process machinery.
 
For the two-periodic  Aztec diamond, there are two types of white vertices and two types of black vertices seen from the two possibilities of edge weights around each white and each black vertex.
To distinguish between these types of vertices, we define for $i\in \{0,1\}$
\begin{equation}
 \mathtt{B}_{i} = \{(x_1,x_2) \in \mathtt{B}: x_1+x_2 \mod 4=2i+1\}
\end{equation}
and 
\begin{equation}
  \mathtt{W}_{i} = \{(x_1,x_2) \in \mathtt{W}: x_1+x_2 \mod 4=2i+1\}.
\end{equation}
There are four different types of dimers having weight $a$ with  $(\mathtt{W}_i,\mathtt{B}_j)$ for $i,j\in\{0,1\}$ and a further four types of dimers having weight $1$ with $(\mathtt{W}_i,\mathtt{B}_j)$ for $i,j\in\{0,1\}$.

The \emph{Kasteleyn matrix} for the two periodic Aztec diamond of size $n=4m$ with  parameters $a$ and $b$, denoted by $K_{a,b}$, is given by 
\begin{equation} \label{pf:K}
      K_{a,b}(x,y)=\left\{\begin{array}{ll}
		     a (1-j) + b j  & \mbox{if } y=x+e_1, x \in \mathtt{B}_j \\
		     (a j +b (1-j) ) \mathrm{i} & \mbox{if } y=x+e_2, x \in \mathtt{B}_j\\
		     a j + b (1-j)  & \mbox{if } y=x-e_1, x \in \mathtt{B}_j \\
		     (a (1-j) +b j ) \mathrm{i} & \mbox{if } y=x-e_2, x \in \mathtt{B}_j\\
			0 & \mbox{if $(x,y)$ is not an edge}
		     \end{array} \right.
\end{equation}
where $\mathrm{i}^2=-1$ and $j\in \{0,1\}$. For the significance of the Kasteleyn matrix for random tiling models, see for example~\cite{Ken:09}.

Since the Aztec diamond graph is   bipartite, meaning that there is a two-coloring of the vertices, from~\cite{Ken:97} the dimers of the two-periodic Aztec diamond form a determinantal point process. More explicitly,  suppose that $E=\{\mathtt{e}_i\}_{i=1}^r$ is a collection of distinct edges with $\mathtt{e}_i=(\mathtt{b}_i,\mathtt{w}_i)$, where $\mathtt{b}_i$ and $\mathtt{w}_i$ denote black and white vertices. 
\begin{thma}[\cite{Ken:97,Joh17}]\label{localstatisticsthm}
    The dimers form a determinantal point process on the edges of the Aztec diamond graph with correlation kernel $L$ meaning that 
	\begin{equation}
		\mathbb{P}(\mathtt{e}_1,\dots, \mathtt{e}_r)=\det L(\mathtt{e}_i,\mathtt{e}_j)_{1 \leq i,j \leq r}
	\end{equation}
where
    \begin{equation}
	  L(\mathtt{e}_i,\mathtt{e}_j) = K_{a,b}(\mathtt{b}_i,\mathtt{w}_i) K_{a,b}^{-1} (\mathtt{w}_j,\mathtt{b}_i).
    \end{equation}
\end{thma}
As mentioned in the introduction, the derivation for the \emph{inverse Kasteleyn matrix}, $K_{a,b}^{-1}$ for the two-periodic Aztec diamond is given in~\cite{CY:13} and a simplification of this formula, which is amenable for asymptotic analysis, is given in~\cite{CJ:16}.  For the purpose of this paper, we set $b=1$.

In order to prove~\eqref{Laplacelimit}, we want to write the expectation on the left side as an expectation of a determinantal point process.  For this, it is convenient to introduce a suitable particle process.  

The space of possible particle positions is $\mathcal{L}_m$ given by~\eqref{Lm}.  To a particle $z \in \mathcal{L}_m$, we associate two vertices $x(z) \in \mathtt{W}$ and $y(z) \in \mathtt{B}$ and the edge $(y(z),x(z))$ between them.  For $z \in \mathcal{L}_m$, and since each $z$ is incident to an $a$-face, we let
\begin{equation}
\begin{split} \label{xyz}
x(z)&= z-\frac{1}{2} (-1)^{\eps(z)} e_2 \\
y(z)&= z+\frac{1}{2} (-1)^{\eps(z)} e_2. \\
\end{split}
\end{equation}
This gives the \emph{particle to edge mapping} 
\begin{equation} \label{particleedge}
\mathcal{L}_m \ni z \longleftrightarrow (y(z),x(z)) \in \mathtt{B} \times \mathtt{W}.
\end{equation}
Using the definitions we see that $x(z) \in \mathtt{W}_{\eps(z)}$ and $y(z) \in \mathtt{B}_{\eps(z)}$; see Fig.~\ref{fig:aface}.
\begin{figure}
\begin{center}
\includegraphics[height=5cm]{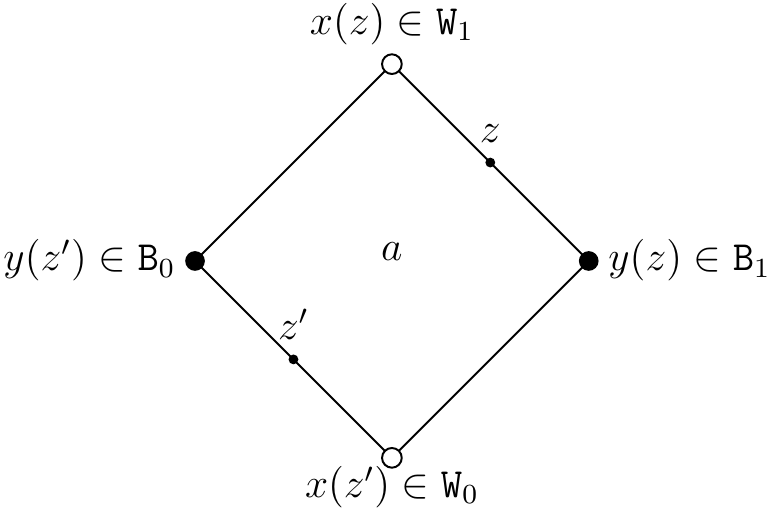}
\caption{An $a$-face on a discrete line with $\eps(z')=0$ and $\eps(z)=1$, $z,z' \in \mathcal{L}_m$.} 
\label{fig:aface}
\end{center}
\end{figure}

From Theorem~\ref{localstatisticsthm}, we know that the dimers, that is the covered edges, form a determinantal point process.  Hence, the mapping~\eqref{particleedge} induces a determinantal point process on $\mathcal{L}_m$.  There is a particle at $z \in \mathcal{L}_m$ if and only if the edge $(y(z),x(z))$ is covered by a dimer.  The next proposition is an immediate consequence of Theorem~\ref{localstatisticsthm}  and the fact that $K_{a,1}(y(z),x(z))=a \mathrm{i} $ for $z \in \mathcal{L}_m$. 

\begin{prop} \label{particlekernel}
The particle process on $\mathcal{L}_m$ defined above is a determinantal point process with correlation kernel $\tilde{K}_m$ given by
\begin{equation}
	\tilde{\mathcal{K}}_m(z,z')=a \mathrm{i} K^{-1}_{a,1} (x(z'),y(z))
\end{equation}
for $z,z' \in \mathcal{L}_m$.
\end{prop}

Recall the definitions of~\eqref{Ipqk} and~\eqref{mu_m}. Let $\{z_i \}$ denote the particle process on $\mathcal{L}_m$ as defined above, and let
\begin{equation}
\mathbbm{I}_{p,q,k}(z) = \left\{\begin{array}{ll}
1 & \mbox{if $z \in I_{p,q,k}$} \\
0 & \mbox{if $z \not \in I_{p,q,k}$} \end{array}  \right.
\end{equation}
be the indicator function for $I_{p,q,k}$. The change in the height function across the interval $I_{p,q,k}$ can be written in terms of the particle process, namely, we have the equation
\begin{equation}
\Delta h(I_{p,q,k})= 4 \sum_{i} (-1)^{\eps(z_i)} \mathbbm{I}_{p,q,k}(z_i),
\end{equation}
where $\sum_i$ is the sum over all particles in the process. From the above equation and~\eqref{mu_m}, we obtain
\begin{equation}
\mu_m(\{\beta_q\} \times A_p ) = \frac{1}{M} \sum_{k=1}^M \sum_i (-1)^{\eps(z_i)} \mathbbm{I}_{p,q,k}(z_i).
\end{equation}
If we let 
\begin{equation}
\psi(z)=\sum_{k=1}^M \sum_{p=1}^{L_2} \sum_{q=1}^{L_1} w_{p,q} (-1)^{\eps(z)} \mathbbm{I}_{p,q,k}(z),
\end{equation}
we see that
\begin{equation}
 \sum_{p=1}^{L_2} \sum_{q=1}^{L_1} w_{p,q} \mu_m(\{\beta_q\} \times A_p) =\frac{1}{M} \sum_i \psi (z_i).
\end{equation}
Since $\{z_i \}$ is a determinantal process on $\mathcal{L}_m$ with correlation kernel $\tilde{K}_m$ we immediately obtain
\begin{equation}
\begin{split} \label{finiteFredholm}
&\mathbb{E}\left[\exp \left(  \sum_{p=1}^{L_2} \sum_{q=1}^{L_1} w_{p,q} \mu_m(\{\beta_q\} \times A_p) \right) \right]  
=\mathbb{E}\left[\exp \left(  \frac{1}{M} \sum_i \psi (z_i) \right) \right]  \\
&=\mathbb{E}\left[\prod_i \left( 1+ \left( e^{\frac{1}{M} \psi(z_i)}-1 \right) \right)   \right] 
	=\det\left(\mathbbm{I} +(e^{\frac{1}{M} \psi} -1)\tilde{\mathcal{K}}_m\right).
\end{split}
\end{equation}
The matrix in the above determinant is indexed by entries of $\mathcal{L}_m$, which is a finite set.
The above formula will be the basis of our asymptotic analysis which will lead to a proof of~\eqref{Laplacelimit}.   To perform this asymptotic analysis, we need some asymptotic formulas which we state in the next section. 

\section{Asymptotic formulas}\label{section:asymptoticformulas}

This section brings forward some of the  key asymptotic results for the liquid-gas boundary from~\cite{CJ:16}. These results are refined specifically for the particle process introduced in Section~\ref{section:particles} and the corresponding scaling associated to $\mathcal{L}_m$.   The origin of these results is made explicit in Section~\ref{sec:proofasymptotics}.

Let
\begin{equation} \label{ctilde}
\tilde{c} (u_1,u_2)=2(1+a^2) + a(u_1+u_1^{-1})(u_2+u_2^{-1}),
\end{equation} 
which is related to the so-called \emph{characteristic polynomial} for the dimer model~\cite{KOS:06}; see~\cite[(4.11)]{CJ:16} for an explanation. Write
\begin{equation}
h(\eps_1,\eps_2)=\eps_1(1-\eps_2)+\eps_2(1-\eps_1).
\end{equation}
We set 
\begin{equation} \label{G}
\mathcal{C}=\frac{1}{\sqrt{2c}} (1-\sqrt{1-2c} ).
\end{equation}

\begin{remark}
  Note that the quantity $\mathcal{C}$ given above is exactly equal to
  the quantity  $|G(\mathrm{i})|$ defined  in~\cite[Eq. (2.6)]{CJ:16},
  that  is  $|G(\mathrm{i})|=\mathcal{C}$.   We  have  simplified  the
  notation since  only $|G(\mathrm{i})|$  appears in  our computations
  here and  the complete definition  of $G$  along with its  choice of
  branch cut is not necessary.
\end{remark}

The full-plane gas phase inverse Kasteleyn matrix is given by
\begin{equation}\label{gasphaseeqn}
\mathbb{K}^{-1}_{1,1}(x,y)=-\frac{\mathrm{i}^{1+h(\eps_x,\eps_y)}}{(2 \pi \mathrm{i})^2} 
\int_{\Gamma_1} \frac{du_1}{u_1}  \int_{\Gamma_1} \frac{du_2}{u_2} \frac{a^{\eps_y} u_2^{1-h(\eps_x,\eps_y)} +a^{1-\eps_y} u_1 u_2^{h(\eps_x,\eps_y)}}{\tilde{c}(u_1,u_2) u_1^{\frac{x_1-y_1+1}{2}} u_2^{\frac{x_2-y_2+1}{2}}},
\end{equation}
where 
$x=(x_1,x_2)\in \mathtt{W}_{\eps_x}$ and $y=(y_1,y_2)\in \mathtt{B}_{\eps_y}$ with $\eps_x,\eps_y \in \{0,1\}$, and $\Gamma_1$ is the positively oriented unit circle; see~\cite[Section 4]{CJ:16} for details.  For the rest of this paper, $\Gamma_R$ denotes a positively oriented circle of radius $R$ around the origin.  
 From~\cite{CJ:16}, it is natural to write
\begin{equation}
K^{-1}_{a,1} (x,y)=\mathbb{K}^{-1}_{1,1}(x,y)-\mathbb{K}_{\mathrm{A}}
\end{equation}
which  defines  $\mathbb{K}_{\mathrm{A}}$.  The  full  expression  for
$\mathbb{K}_{\mathrm{A}}$   is   complicated,  see~\cite[Theorem 2.3]{CJ:16}.   Its
asymptotics  is  given  in  Section~\ref{sec:proofasymptotics}.  Since
$K_{a,1}(x(z),y(z))=a \mathrm{i}$, this leads us to define
\begin{equation} \label{Ktilde}
\begin{split}
\tilde{\mathcal{K}}_{m,0}(z,z')&=a \mathrm{i} \mathbb{K}^{-1}_{1,1} (x(z'),y(z)) \\
\tilde{\mathcal{K}}_{m,1}(z,z')&=a \mathrm{i} \mathbb{K}_{\mathrm{A}} (x(z'),y(z)) \\
\end{split}
\end{equation}
so that
\begin{equation}
\tilde{\mathcal{K}}_{m}(z,z')=\sum_{\delta \in \{0,1\}}(-1)^{\delta} \tilde{\mathcal{K}}_{m,\delta}(z,z').
\end{equation}
For $z \in \mathcal{L}_m(q,k)$, we define 
\begin{equation}
\begin{split} \label{gi}
\gamma_1 (z)& = \frac{t(z)}{\lambda_1 (2m)^{1/3}} \beta_q -\frac{1}{3} \beta_q^3, \\
\gamma_2(z) & = \eps(z) +\beta_m(q,k), \\
\gamma_3(z) &= 2(t(z)-\tau_m(q))+\beta_m(q,k).
\end{split}
\end{equation}
We also introduce the relation 
\begin{equation}
\mathtt{g}_{\eps_1,\eps_2} =
\left\{
\begin{array}{ll}
\frac{\mathrm{i} \left(\sqrt{a^2+1}+a\right)}{1-a}
 &  \mbox{if } (\eps_1,\eps_2)=(0,0) \\
\frac{\sqrt{a^2+1}+a-1}{\sqrt{2a} (1-a) }
 & \mbox{if } (\eps_1,\eps_2)=(0,1) \\
-\frac{\sqrt{a^2+1}+a-1}{\sqrt{2a} (1-a) }
&\mbox{if } (\eps_1,\eps_2)=(1,0)\\
\frac{\mathrm{i}\left(\sqrt{a^2+1}-1\right)}{(1-a) a}
&\mbox{if } (\eps_1,\eps_2)=(1,1).
\end{array}
\right.
\end{equation}
Let $z,z' \in \mathcal{L}_m$ and write $x(z')=(x_1(z'),x_2(z'))$ and $y(z)=(y_1(z),y_2(z))$.  Motivated by the asymptotic results from~\cite{CJ:16}, compare~\cite[Theorem 2.7]{CJ:16}, we define for $\delta = 0,1$,
\begin{equation}
\label{Kmdelta}
 \mathcal{K}_{m,\delta} (z,z') =\frac{1}{\lambda_1 c_0 a \mathrm{i} \mathtt{g}_{\eps(z'),\eps(z)}}  \mathrm{i}^{y_1(z)-x_1(z')-1} e^{\gamma_1(z') -\gamma_1(z)}  \mathcal{C}^{\frac{1}{2}(2+x_1(z')-x_2(z')+y_2(z)-y_1(z))} \tilde{\mathcal{K}}_{m,\delta}(z,z') 
\end{equation}
and 
\begin{equation} \label{Kdeltaeqn}
\mathcal{K}_m(z,z')=\sum_{\delta \in \{0,1\}} (-1)^{\delta} \mathcal{K}_{m,\delta}(z,z').
\end{equation} 
$\mathcal{K}_{m,\delta} (z,z')$ is the object that will have nice scaling limits and that we can control as $m\to\infty$, see Proposition~\ref{Asymptotics} below.
If $z \in \mathcal{L}_m(q,k)$ and $z' \in \mathcal{L}_m(q',k')$, a computation using~\eqref{tz},~\eqref{xyz} and~\eqref{gi} gives
\begin{equation}\label{xyg2}
\frac{2+x_1(z')-x_2(z')+y_2(z)-y_1(z)}{2} = \gamma_2(z')-\gamma_2(z)+2-2\eps(z')
\end{equation}
and
\begin{equation} \label{xyg3}
y_1(z)-x_1(z')-1=\gamma_3(z)-\gamma_3(z')+2\eps(z').
\end{equation}
Applying these formulas in~\eqref{Kmdelta} and using the fact that $\lambda_1 c_0 =\frac{1}{2} \sqrt{1-2c}$ we obtain
\begin{equation}
\begin{split} \label{Ktilde2}
\tilde{\mathcal{K}}_{m,\delta}(z,z')&= \frac{a \mathrm{i}}{2} \sqrt{1-2c} \mathtt{g}_{\eps(z'),\eps(z)} \mathcal{C}^{2 \eps(z')-2} (-1)^{\eps(z')} e^{\gamma_1(z') -\gamma_1(z)} \\
& \times
 \mathcal{C}^{\gamma_2(z) -\gamma_2(z')} \mathrm{i}^{\gamma_3(z')-\gamma_3(z)} \mathcal{K}_{m,\delta}(z,z').
\end{split}
\end{equation}
From this, we see that
\begin{equation} \label{KmKm}
\mathtt{K}_m(z,z') =  \frac{a \mathrm{i}}{2} \sqrt{1-2c} \mathtt{g}_{\eps(z'),\eps(z)} \mathcal{C}^{2 \eps(z')-2}(-1)^{\eps(z')}\mathcal{K}_{m}(z,z')
\end{equation}
is also a correlation kernel for the particle process on $\mathcal{L}_m$. See Section~\ref{sec:proofasymptotics} for specific signposting of where these formulas come from. 

The next proposition contains the asymptotic formulas and estimates that we will need in the proof of our main theorem.   The proof will be given in Section~\ref{sec:proofasymptotics}.
\begin{prop} \label{Asymptotics}
Let $z \in \mathcal{L}_m(q,k)$, $z \in \mathcal{L}_m(q',k')$ and write $t=t(z)$, $t'=t(z')$.  Consider $\mathcal{K}_{m,\delta}(z,z')$ defined by~\eqref{Kmdelta}. 
The asymptotic formulas and estimates below are uniform as $m\to\infty$ for $|t|,|t'| \leq C(2m)^{1/3}$, for any fixed $C>0$ and $1 \leq k \leq M$. 
\begin{enumerate}
\item For any $q,q'$,
\begin{equation}
\mathcal{K}_{m,1}(z,z') = \frac{1}{\lambda_1(2m)^{1/3}} \tilde{\mathcal{A}}\left(\beta_{q'},\frac{t'}{\lambda_1(2m)^{1/3}} ; \beta_{q},\frac{t}{\lambda_1(2m)^{1/3}} \right)(1+o(1)).
\end{equation}
		where $\tilde{\mathcal{A}}$ is given in~\eqref{eq:Airymod}.
\item If $q\not=q'$, then 
\begin{equation}
\mathcal{K}_{m,0}(z,z') = \frac{1}{\lambda_1(2m)^{1/3}} \phi_{\beta_{q'},\beta_q} \left( \frac{t'}{\lambda_1(2m)^{1/3}}, \frac{t}{\lambda_1(2m)^{1/3}} \right)(1+o(1))
\end{equation}
		where $c_1>0$ is a constant adn $\phi_{\beta_{q'},\beta_q}$ is given in~\eqref{eq:Airyphi}.
\item Assume that $q=q'$ and $k>k'$. Then there are constants $c_1,c_2,C>0$, so that
\begin{enumerate}
\item 
\begin{equation}
\mathcal{K}_{m,0}(z,z')= \frac{1}{\lambda_1(\log m)} \frac{1}{\sqrt{4 \pi (k-k')}} \exp \left(-\frac{1}{4(k-k')} \left( \frac{t'-t}{\lambda_1 \log m }\right)^2 \right) (1+o(1))
\end{equation}
if $|t'-t| \leq c_2( (k-k')( \log m )^2)^{7/12}$,
\item 
\begin{equation}
|\mathcal{K}_{m,0}(z,z')| \leq\frac{C}{(\log m)\sqrt{k-k'}} \exp\left(-\frac{c_1}{(k-k')} \left( \frac{t'-t}{\lambda_1 \log m }\right)^2 \right) 
\end{equation}
if $ c_2( (k-k')( \log m )^2)^{7/12} \leq  |t'-t| \leq \lambda_2 (k-k') (\log m)^2$,
\item 
and
\begin{equation}
|\mathcal{K}_{m,0}(z,z')| \leq C e^{-c_1(k-k')(\log m)^2} 
\end{equation}
if $|t'-t| \geq \lambda_2 (k-k') (\log m)^2$.
\end{enumerate}

\item Assume that $q=q'$ and $k<k'$. Then there are constants $c_1,C>0$ so that
\begin{equation}
|\mathcal{K}_{m,0}(z,z')| \leq C e^{-c_1(k'-k)(\log m)^2}.
\end{equation}
\item  Assume that $q=q'$ and $k=k'$.  Then there are constants $c_1,C>0$ so that
\begin{equation}
|\mathcal{K}_{m,0}(z,z')| \leq C e^{-c_1|t'-t|}.
\end{equation}

\end{enumerate}
\end{prop}

\section{Proof of Theorem~\ref{thm:moments}} \label{sec:mainthmproof}

In this section, we give the proof of Theorem~\ref{thm:moments} relying on Proposition~\ref{Asymptotics} and Lemmas~\ref{U0} and~\ref{Tepsilon} whose proofs are deferred to later in the paper.  To prove Theorem~\ref{thm:moments}, we analyze the right side of~\eqref{finiteFredholm} via its cumulant or trace expansion.  Since $\mathtt{K}_m$, given by (\ref{KmKm}), is also a correlation kernel for the particle process, we have
\begin{equation}
\det ( \mathbbm{I}  + (e^{\frac{1}{M} \psi}-1)\tilde{\mathcal{K}}_m) = \det ( \mathbbm{I}  + (e^{\frac{1}{M} \psi}-1)\mathtt{K}_m). 
\end{equation}
For $|w_{p,q}| \leq R$ with $R$ sufficiently small, we have the expansion 
\begin{equation}
\begin{split} \label{traceexp}
&\log \det ( \mathbbm{I}  + (e^{\frac{1}{M} \psi}-1)\mathtt{K}_m)_{\mathcal{L}_m)} \\
&= \sum_{s=1}^\infty \frac{1}{M^s} \sum_{r=1}^s \frac{(-1)^{r+1}}{r} \sum_{\substack{\ell_1+\dots+\ell_r=s \\ \ell_1,\dots,\ell_r \geq 1}} \frac{1}{\ell_1! \dots \ell_r!} \mathrm{tr} \left( \psi^{\ell_1} \mathtt{K}_m \dots \psi^{\ell_r} \mathtt{K}_m \right).
\end{split}
\end{equation}
For a simple proof of this expansion, see e.g.\ p.\ 450 in~\cite{BD:14}.
Since $\mathcal{L}_m$ is finite we have a finite-dimensional operator, and the expansion is convergent if $R$ is small enough.
Note that \emph{a priori}, $R$ could depend on $m$. It is a consequence of the proof below that we are able to choose $R$ independent of $m$. 

Since all the discrete intervals $I_{p,q,k}$ have disjoint support,
\begin{equation}\label{eq:psiexpand}
\psi(z)^\ell = \sum_{k=1}^M \sum_{p=1}^{L_2} \sum_{q=1}^{L_1} w^\ell_{p,q} (-1)^{\ell \eps (z)} \mathbbm{I}_{p,q,k}(z),
\end{equation}
for all $l\geq 1$. 
In what follows, we use the notation $\overline{j} \in S^r$ to denote the sum over all $j_1,\dots, j_r \in S$ for some set $S$ and we assume the notation to be cyclic with respect to $r$, that is $j_{r+1}=j_1$. Also, we use the notation $[N]=\{1,\dots,N\}$. Thus, we have from~\eqref{eq:psiexpand}
\begin{equation}
\begin{split} \label{trace1}
\mathrm{tr} \left( \psi^{\ell_1} \mathtt{K}_m \dots \psi^{\ell_r} \mathtt{K}_m \right)
&= \sum_{\overline{z} \in (\mathcal{L}_m)^r}  \sum_{\overline{k} \in [M]^r} \sum_{\overline{p} \in [L_2]^r} \sum_{\overline{q} \in [L_1]^r}
\prod_{i=1}^rw_{p_i,q_i}^{\ell_i} (-1)^{\ell_i \eps(z_i)}   \mathbbm{I}_{p_i,q_i,k_i}(z_i) \mathtt{K}_m(z_i,z_{i+1}) \\
&= \sum_{\overline{\eps} \in\{0,1\}^r}   \sum_{\overline{k} \in [M]^r} \sum_{\overline{p} \in [L_2]^r} \sum_{\overline{q} \in [L_1]^r}
\prod_{i=1}^rw_{p_i,q_i}^{\ell_i} (-1)^{\ell_i \eps_i}
\sum_{\overline{z} \in (\mathcal{L}_m)^r}
\prod_{i=1}^r 
\mathbbm{I}_{p_i,q_i,k_i}^{\eps_i} (z_i)
  \mathtt{K}_m(z_i,z_{i+1}). \\
\end{split}
\end{equation} 
Here, $\mathbbm{I}_{p,q,k}^{\eps}$ is the indicator function on $\mathcal{L}_m$ for the set
\begin{equation}
{I}_{p,q,k}^{\eps}= \{z \in I_{p,q,k} ; \eps(z)=\eps \}
\end{equation}
for $\eps \in \{0,1\}$. Write $\mathtt{K}_m = \sum_{\delta \in \{0,1\}}(-1)^{\delta} \mathtt{K}_{m,\delta}$, similarly to~\eqref{Kdeltaeqn}, and plug it into~\eqref{trace1} to get
\begin{equation} \label{trace2}
\begin{split}
\mathrm{tr} \left( \psi^{\ell_1} \mathtt{K}_m \dots \psi^{\ell_r} \mathtt{K}_m \right)
&= \sum_{\overline{\eps},\overline{\delta} \in\{0,1\}^r}  \prod_{i=1}^r (-1)^{\ell_i \eps_i} 
  \sum_{\overline{k} \in [M]^r} \sum_{\overline{p} \in [L_2]^r} \sum_{\overline{q} \in [L_1]^r}
\prod_{i=1}^rw_{p_i,q_i}^{\ell_i} (-1)^{\delta_i} \\
&\times \sum_{\overline{z} \in (\mathcal{L}_m)^r}
\prod_{i=1}^r 
\mathbbm{I}_{p_i,q_i,k_i}^{\eps_i} (z_i)
  \mathtt{K}_{m,\delta_i}(z_i,z_{i+1}). \\ 
\end{split}
\end{equation}
In order to carry out the asymptotic analysis, we will split this trace into four parts. Let 
\begin{equation}
D_r=\{0,1\}^r \times [M]^r \times [L_2]^r \times [L_1]^r.
\end{equation}
Define
\begin{equation}
	D_{r,0} = \{(\overline{\delta},\overline{k},\overline{p},\overline{q}) \in D_r; \delta_i =0,k_i=k_{i+1},p_i=p_{i+1} \mbox{~and~} q_i=q_{i+1},1\leq i \leq r\},
\end{equation}
\begin{equation}
D_{r,1} = \{ (\overline{\delta},\overline{k},\overline{p},\overline{q}) \in D_r; \delta_i =0, q_i=q_{i+1} \mbox{~for~}1\leq i \leq r \mbox{~and~}p_i\not=p_{i+1} \mbox{~for some~} i \},  
\end{equation}
\begin{equation}
D_{r,2} = \{ (\overline{\delta},\overline{k},\overline{p},\overline{q}) \in D_r; \delta_i =0,q_i=q_{i+1},p_i=p_{i+1}\mbox{~for~}1\leq i\leq r  \mbox{~and~} k_i\not=k_{i+1} \mbox{~for some~}i\},
\end{equation}
and
\begin{equation}
D_{r,3} = \{ (\overline{\delta},\overline{k},\overline{p},\overline{q}) \in D_r; \delta_i =1 \mbox{~or~} q_i \not = q_{i+1} \mbox{~for some~} i \}.
\end{equation}
Then, we have $D_r=D_{r,0} \cup D_{r,1} \cup D_{r,2} \cup D_{r,3}$. Introduce
\begin{equation}  \label{Tjmrell}
T_j(m,r,\overline{l}) = \sum_{\overline{\eps} \in \{0,1\}^r}  \prod_{i=1}^r (-1)^{\ell_i \eps_i} \sum_{(\overline{\delta},\overline{k},\overline{p},\overline{q}) \in D_{r,j}} \prod_{i=1}^r w_{p_i,q_i}^{\ell_i} (-1)^{\delta_i}  
\sum_{\overline{z} \in( \mathcal{L}_m)^r } \prod_{i=1}^r \mathbbm{I}_{p_i,q_i,k_i}^{\eps_i} (z_i)
  \mathtt{K}_{m,\delta_i}(z_i,z_{i+1}),
\end{equation}
for  $0\leq j \leq 3$.  Then, by~\eqref{traceexp} and~\eqref{trace2} we have
\begin{equation}
\log \det( \mathbbm{I} +(e^{\frac{1}{M} \psi}-1) \mathtt{K}_m) = \sum_{j=0}^3 U_j(m)
\end{equation}
where we define
\begin{equation} \label{Ujm}
U_j(m)=  \sum_{s=1}^\infty \frac{1}{M^s} \sum_{r=1}^s \frac{(-1)^{r+1}}{r} \sum_{\substack{\ell_1+\dots+\ell_r=s \\ \ell_1,\dots,\ell_r \geq 1}} \frac{T_j(m,r,\overline{\ell})}{\ell_1! \dots \ell_r!}.
\end{equation}
Our goal is now to show that $U_j(m)$ tends to zero as $m$ tends  infinity for $j=0,1,2$ and then to compute the limit of $U_3(m)$, which will give us what we want. The proof of $U_0(m)$ tends to zero as $m$ tends to infinity is rather involved and requires a separate argument.  We formulate it as a lemma but postpone the proof until Section~\ref{sec:U0proof}.

\begin{lemma} \label{U0}
There is an $R>0$ such that $\lim_{m\to \infty}U_0(m)=0$ uniformly for $|w_{p,q}| \leq R$. 
\end{lemma}

Recall~\eqref{KmKm} and define
\begin{equation} \label{Pepsilonell}
P(\overline{\eps},\overline{\ell}) =\prod_{i=1}^r \frac{a\mathrm{i}}{2} \sqrt{1-2c} (-1)^{(1+\ell_i) \eps_i} \mathtt{g}_{\eps_i,\eps_{i+1}} \mathcal{C}^{\eps_i+\eps_{i+1} -2}.
\end{equation}
Then, we have 
\begin{equation}
T_j(m,r,\overline{\ell}) =\sum_{\overline{\eps} \in \{0,1\}^r} P(\overline{\eps},\overline{\ell})
 \sum_{(\overline{\delta},\overline{k},\overline{p},\overline{q}) \in D_{r,j}} \prod_{i=1}^r w_{p_i,q_i}^{\ell_i} (-1)^{\delta_i}  
\sum_{\overline{z} \in( \mathcal{L}_m)^r } \prod_{i=1}^r \mathbbm{I}_{p_i,q_i,k_i}^{\eps_i} (z_i)
  \mathcal{K}_{m,\delta_i}(z_i,z_{i+1}),
\end{equation}
for $j=1,2,3$. 

Recall~\eqref{tz},~\eqref{zqks},~\eqref{Apmtilde} and~\eqref{Ipqk}.  Define 
\begin{equation}
z_{q,k}^{\eps} (t) =\left(\rho_m +2(t-\tau_m(q)) -\eps+\frac{1}{2} \right) e_1 -\beta_m(q,k)e_2 
\end{equation}
and
\begin{equation} \label{Apm}
A_{p,m} = \{t \in \mathbb{Z}; [\alpha^l_p \lambda_1 (2m)^{1/3}] \leq t \leq [\alpha^r_p \lambda_1(2m)^{1/3}] \},
\end{equation} where we recall the notation $A_p=[\alpha_p^l,\alpha_p^r]$ for all $1 \leq p \leq L_2$. 
Then, we can write
\begin{equation}
I_{p,q,k}^\eps = \{z_{q,k}^{\eps}(t) ; t \in A_{p,m} \}.
\end{equation}
Hence, we can also write
\begin{equation}
\begin{split} \label{Sr}
S_r (\overline{\eps}, \overline{\delta} ,\overline{k} ,\overline{p}, \overline{q}) &:= 
\sum_{\overline{z} \in (\mathcal{L}_m)^r} \prod_{i=1}^r \mathbbm{I}_{p_i,q_i,k_i}^{\eps_i} (z_i) \mathcal{K}_{m,\delta_i} (z_i,z_{i+1})  \\
&= \sum_{\overline{t} \in \mathbb{Z}^r} \prod_{i=1}^r \mathbbm{I}_{A_{p_i,m}} (t_i) \mathcal{K}_{m,\delta_i} (z_{q_i,k_i}^{\eps_i}(t_i), z_{q_{i+1},k_{i+1}}^{\eps_{i+1}}(t_{i+1} )) \\
&= \int_{\mathbb{R}^r} d^r \overline{t} \prod_{i=1}^r \mathbbm{I}_{A_{p_i,m}} ([t_i])\mathcal{K}^{(i)}_{m,\overline{\eps}, \overline{\delta}, \overline{k}, \overline{q}}(t_i,t_{i+1}) 
\end{split}
\end{equation}
where 
\begin{equation}
\mathcal{K}^{(i)}_{m,\overline{\eps}, \overline{\delta}, \overline{k}, \overline{q}}(t,t') = 
\mathcal{K}_{m,\delta_i} (z_{q_i,k_i}^{\eps_i}([t]), z_{q_{i+1},k_{i+1}}^{\eps_{i+1}}([t'] )).
\end{equation} 
With this notation, we see that
\begin{equation} \label{Tj}
T_j(m,r,\overline{\ell}) =\sum_{\overline{\eps} \in \{0,1\}^r} P(\overline{\eps},\overline{\ell})
 \sum_{(\overline{\delta},\overline{k},\overline{p},\overline{q}) \in D_{r,j}} \prod_{i=1}^r w_{p_i,q_i}^{\ell_i} (-1)^{\delta_i} S_r (\overline{\eps}, \overline{\delta} ,\overline{k} ,\overline{p}, \overline{q}) 
\end{equation}
for $j=1,2,3$.

\begin{lemma} \label{Ujsmall} 
There is an $R>0$ such that, for $j=1,2$, $\lim_{m \to \infty} U_j(m)=0$ uniformly in $|w_{p,q}| \leq R$.

\end{lemma} 
\begin{proof}
Consider $j=1$ so that $(\overline{\delta},\overline{k},\overline{p},\overline{q}) \in D_{r,1}$.   There is an $i=i_1$ such that $p_{i_1} \not =p_{i_1+1}$ by the definition of $D_{r,1}$.  We have $\delta_i=0$ for all $i$.  Hence, by statements (3) to (5) in Proposition~\ref{Asymptotics},  we have
\begin{equation}
|\mathcal{K}^{(i_1)}_{m,\overline{\eps}, \overline{\delta}, \overline{k}, \overline{q}}(t_{i_1},t_{i_1+1}) | \leq  C e^{-c_1 (\log m)^2} 
\end{equation}
since  $|t_{i_1+1}-t_{i_1}| \geq  C m^{1/3}$  --- note  that the  real
estimate in  the above inequality  is actually  less than or  equal to
$Ce^{-c_1m^{1/3}}$  but  we do  not  need  this  here. All  the  other
$\mathcal{K}^{(i)}$  factors in  the  integrand  in~\eqref{Sr} can  be
estimated      using      statements      (3)      to      (5)      in
Proposition~\ref{Asymptotics}; to make this  argument very precise, we
can use the same type of change of variables (\ref{ttauchange}) below,
we omit the details. From this, we see that
\begin{equation}
|S_r (\overline{\eps}, \overline{\delta} ,\overline{k} ,\overline{p}, \overline{q})| \leq C^r m^{2/3} e^{-c_1(\log m)^2}.
\end{equation}
Consequently, by~\eqref{Tj}, since $|P(\overline{\eps},\overline{\ell})| \leq C^r$,
\begin{equation}
|T_j(m,r,\overline{\ell})| \leq C^r M^r R^s m^{2/3} e^{-c_1 (\log m)^2}.
\end{equation}
We can use this estimate in~\eqref{Ujm} to see that 
\begin{equation}
\begin{split}
| U_1(m) | &\leq \sum_{s=1}^{\infty} \frac{R^s}{M^s} \sum_{r=1}^s \frac{1}{r} \sum_{\substack{\ell_1+\dots + \ell_r=s \\ \ell_1,\dots,\ell_r \geq 1}} \frac{(CM)^r m^{2/3} e^{-c_1 (\log m)^2}}{\ell_1! \dots \ell_r!}  \\
& \leq C m^{2/3} e^{-c_1 (\log m)^2} \sum_{s=1}^{\infty} (CR)^s \leq C m^{2/3}e^{-c_1 (\log m)^2}
\end{split}
\end{equation}
provided that $R$ is small enough. Here, we used the fact that
\begin{equation}
\sum_{\substack{\ell_1+\dots + \ell_r=s \\ \ell_1,\dots,\ell_r \geq 1}} \frac{1}{\ell_1! \dots \ell_r!} \leq \left( \sum_{\ell=0}^{\infty} \frac{1}{\ell!} \right)^r =e^r.
\end{equation}

We next consider $j=2$ so that $(\overline{\delta},\overline{k},\overline{p},\overline{q}) \in D_{r,2}$.  We cannot have $k_i >k_{i+1}$ for all $i$ since it violates the cyclic condition.  Hence, when estimating the $\mathcal{K}^{(i)}$ in the integrand in~\eqref{Sr}, we have to use statement (4) in Proposition~\ref{Asymptotics} at least once.  We now proceed in exactly the same way as above to prove that $U_2(m) \to 0$ as $m \to \infty$. 
\end{proof}

It remains to consider $U_3(m)$.  This means that we need to control $S_r(\overline{\eps},\overline{\delta},\overline{k},\overline{p},\overline{q})$ in the case when $(\overline{\delta},\overline{k},\overline{p},\overline{q}) \in D_{r,3}$.  There are two sub-cases: for a given $(\overline{\delta},\overline{k},\overline{p},\overline{q}) \in D_{r,3}$, 
\begin{enumerate}
\item if $\delta_i=1$ for some $i$, we define $i_1$, by $\delta_1=\dots=\delta_{i_1-1}=0$, $\delta_{i_1}=1$,  
\item if $\delta_i=0$ for all $i$, we define $i_1$ by $q_1= \dots = q_{i_1} \not = q_{i_1+1}$.  
\end{enumerate}
Such $i_1$'s always exist by the definition of $D_{r,3}$. Define $d_i$, $1 \leq i \leq r$, by 
\begin{equation}
d_i = \left\{
\begin{array}{ll}
\lambda_1 (2m)^{1/3} & \mbox{if $q_{i} \not = q_{i+1}$ or $\delta_i=1$} \\
(\log m)\lambda_1 \sqrt{|k_{i+1}-k_i|} &\mbox{if $\delta_i=0, q_{i} = q_{i+1}, k_i \not = k_{i+1}$ } \\
1 & \mbox{if $\delta_i=0,q_i=q_{i+1}, k_i=k_{i+1}$.}
\end{array}
\right.
\end{equation}
We now  introduce new coordinates in~\eqref{Sr} by
\begin{equation} \label{ttauchange}
 \left\{
\begin{array}{ll}
\tau_{i_1} =t_{i_1}/d_{i_1} \\ 
\tau_{i} =( t_{i+1}-t_i)/d_{i}  & \mbox{if $i>i_{1}$}, 
\end{array}
\right.
\end{equation}
recalling that the indices are cyclic.  The inverse transformation is 
\begin{equation}
t_i=t_i(\overline{\tau}) = \sum_{j=i_1}^i d_j \tau_j
\end{equation}
for $i_1 \leq i < i_1+r$. After this change of variables, we obtain
\begin{equation}
S_r (\overline{\eps}, \overline{\delta} ,\overline{k} , \overline{q}) = \int_{\mathbb{R}^r} d^r \overline{\tau} \prod_{i=1}^r \mathbbm{I}_{A_{{p_i},m}} ([t_i(\overline{\tau})]) d_i  \mathcal{K}^{(i)}_{m,\overline{\eps},\overline{\delta},\overline{k},\overline{q} } (t_i (\overline{\tau}), t_{i+1} (\overline{\tau})) 
\end{equation}
The next lemma gives a bound on $S_r$. 

\begin{lemma} \label{Srbound}
There is a constant $C>0$ such that  
\begin{equation}
|S_r (\overline{\eps}, \overline{\delta} ,\overline{k} ,\overline{p}, \overline{q}) | \leq C^r
\end{equation}
for all $(\overline{\eps}, \overline{\delta} ,\overline{k} , \overline{q}) \in D_{r,3}$ and $\overline{\eps} \in \{0,1\}^r$.  

\end{lemma}

\begin{proof}
If $\delta_i =1$, then statement (1) in Proposition~\ref{Asymptotics} gives
\begin{equation}
|  d_i  \mathcal{K}^{(i)}_{m,\overline{\eps},\overline{\delta},\overline{k},\overline{q} } (t_i (\overline{\tau}), t_{i+1} (\overline{\tau})) | \leq C \left|\tilde{\mathcal{A} }\left(\beta_{q_i} , \frac{t_i (\overline{\tau})}{\lambda_1 (2m)^{1/3}} ; \beta_{q_{i+1}} , \frac{t_{i+1} (\overline{\tau}) }{\lambda_1 (2m)^{1/3}} \right)\right|.
\end{equation}
Similarly, if $\delta_i =0$, $q_i\not = q_{i+1}$, then 
\begin{equation}
|  d_i  \mathcal{K}^{(i)}_{m,\overline{\eps},\overline{\delta},\overline{k},\overline{q} } (t_i (\overline{\tau}), t_{i+1} (\overline{\tau}))| \leq C \left| 
\phi_{\beta_{q_i}, \beta_{q_{i+1}}} \left(  \frac{t_i (\overline{\tau})}{\lambda_1 (2m)^{1/3}}   , \frac{t_{i+1} (\overline{\tau}) }{\lambda_1 (2m)^{1/3}} \right)\right|,
\end{equation}
by (2) in Proposition~\ref{Asymptotics}. Furthermore, we obtain the following estimates  for $\delta_i=0$ and $q_i= q_{i+1}$.
\begin{itemize}
\item If  $k_i >k_{i+1}$ and $|\tau_i|  \leq  c_2 ( (k_i-k_{i+1}) (\log m)^2)^{1/2}$, then  
\begin{equation}
|  d_i  \mathcal{K}^{(i)}_{m,\overline{\eps},\overline{\delta},\overline{k},\overline{q} } (t_i (\overline{\tau}), t_{i+1} (\overline{\tau}))| \leq Ce^{-c_1'\tau_i^2},
\end{equation}
where $c_1'>0$, which follows from statement (3)(a) in Proposition~\ref{Asymptotics}.
\item If $k_i <k_{i+1}$,  or $k_i >k_{i+1}$  and $|\tau_i|  >  c_2 ( |k_i-k_{i+1}| (\log m)^2)^{1/2}$, then 
\begin{equation}
|  d_i  \mathcal{K}^{(i)}_{m,\overline{\eps},\overline{\delta},\overline{k},\overline{q} } (t_i (\overline{\tau}), t_{i+1} (\overline{\tau}))| \leq Ce^{-c_1|k_i-k_{i+1}| (\log m)^2},
\end{equation}
which follows from statements (3)(a), (3)(b) and (4) in Proposition~\ref{Asymptotics}.
\item If $k_i =k_{i+1}$, then
\begin{equation}
|  d_i  \mathcal{K}^{(i)}_{m,\overline{\eps},\overline{\delta},\overline{k},\overline{q} } (t_i (\overline{\tau}), t_{i+1} (\overline{\tau}))| \leq Ce^{-c_1|\tau_i|},
\end{equation}
which follows from statement (5) in Proposition~\ref{Asymptotics}.
\end{itemize}
If we use these estimates and the fact that $|t_i| \leq Cm^{1/3}$ for all $1 \leq i \leq r $, we get the bound on $S_r$.

\end{proof}

We can now prove that we have a uniform control of the series defining $U_3(m)$.

\begin{lemma} \label{Uniform} 
The series~\eqref{Ujm} defining $U_3(m)$ is uniformly convergent for $|w_{p,q}| \leq R$ if $R$ is sufficiently small. 
\end{lemma}

\begin{proof}
It follows from~\eqref{Ujm},~\eqref{Tj} and the bound in Lemma~\ref{Srbound} that
\begin{equation}
|U_3(m)| \leq \sum_{s=1}^{\infty} \frac{1}{M^s} \sum_{r=1}^s \frac{1}{r} \sum_{\substack{\ell_1+\dots+\ell_r=s \\ \ell_1,\dots,\ell_r \geq 1}} \sum_{\overline{\eps} \in \{0,1\}^r} \frac{|P(\overline{\eps},\overline{\ell})|}{\ell_1!\dots \ell_r! } \sum_{( \overline{\delta} ,\overline{k} ,\overline{p}, \overline{q})  \in D_{r,3}} R^s C^r 
\leq \sum_{s=1}^{\infty} \frac{R^s}{M^s} \sum_{r=1}^s (CM)^r <\infty 
\end{equation}
if $|w_{p,q}| \leq R$, and $R$ is sufficiently small. 
\end{proof} 

Let 
\begin{equation}
D_{s,3}^*=   \{(    \overline{\delta}   ,\overline{k}   ,\overline{p},
\overline{q}) \in D_{s,3}; k_i \not = k_j \mbox{~for all $i \not = j$} \}
\end{equation}
and write 
\begin{equation}\label{PtoQ}
Q(\overline{\eps}) = P(\overline{\eps},(1,\dots, 1))
\end{equation}
 with the vector $(1,\dots,1)$ having length $s$. Define
\begin{equation}
U_3^* (m) =\sum_{s=1}^\infty \frac{(-1)^{s+1}}{s M^s} \sum_{\overline{\eps} \in \{0,1\}^s} Q(\overline{\eps}) \sum_{( \overline{\delta} ,\overline{k} ,\overline{p}, \overline{q})  \in D_{s,3}^*} \prod_{i=1}^s (-1)^{\delta_i} w_{p_i,q_i} S_s(\overline{\eps} , \overline{\delta} ,\overline{k} ,\overline{p}, \overline{q}). \label{Ustar3m}
\end{equation}

\begin{lemma} \label{U3mstar} 
There is a constant $C>0$ so that for $|w_{p,q}| \leq R$, with $R$ sufficiently small,
\begin{equation}
\left| U_3(m) - U_3^*(m) \right| \leq \frac{C}{M}.
\end{equation}
\end{lemma} 

\begin{proof}
The same argument as in the proof of the previous lemma shows that
\begin{equation}
\begin{split}
&\left| \sum_{s=1}^{\infty} \frac{1}{M^s} \sum_{r=1}^{s-1} \frac{(-1)^{r+1}}{r} \sum_{\substack{\ell_1+\dots+\ell_r=s \\ \ell_1,\dots,\ell_r \geq 1}} \sum_{\overline{\eps} \in \{0,1\}^r} \frac{P(\overline{\eps},\overline{\ell})}{\ell_1!\dots \ell_r! } \sum_{( \overline{\delta} ,\overline{k} ,\overline{p}, \overline{q})  \in D_{r,3}}
S_r(\overline{\eps} , \overline{\delta} ,\overline{k} ,\overline{p}, \overline{q}) \right| \\ &\leq 
\sum_{s=1}^\infty \frac{R^s}{M^s} \sum_{r=1}^{s-1} \frac{C^r M^r}{r} \leq \frac{C}{M}.
\end{split}
\end{equation}
If $r=s$, and $k_i=k_j$ for some $i,j$, then the number of elements in $D_{s,3}$ is less than $C M^{s-1}$ and we use the same estimates as used in the proof of the previous lemma.
\end{proof}

Since $M$ tends to infinity (slowly) as $m$ tends to infinity, we only have to consider $U_3^*(m)$.  

Given $( \overline{\delta} ,\overline{k} ,\overline{p}, \overline{q})  \in D^*_{s,3} $ we let $1 \leq j_1 < \dots < j_r \leq s$ be the indices $i$ where $\delta_i=1$, or $\delta_i=0$ and $q_i \not = q_{i+1}$.  Let $\ell_1=j_1-j_r+s$, $\ell_2 = j_2-j_1$, $\dots$, $\ell_r = j_r-j_{r-1}$.  We see that $\ell_i \geq 1$ and $\ell_1+\dots+\ell_r=s$.  Also, $j_r =j_1+s-\ell_1 \leq s$, which implies that $j_1 \leq \ell_1$.  Hence, given $\ell_1,\dots,\ell_r$ with $\ell_1+\dots +\ell_r=s$, $\ell_i \geq 1$  for all $1 \leq i \leq r$, and $j_1$ with $1 \leq j_1 \leq \ell_1$, we can uniquely reconstruct $j_1,\dots, j_r$.  

Write $J=\{j_1,\dots, j_r \}$ and $J'=[s] \backslash J$.  Then, using~\eqref{Sr}, we have
\begin{equation}
\begin{split}
S_s (\overline{\eps}, \overline{\delta} ,\overline{k} ,\overline{p}, \overline{q})  
&= \int_{\mathbb{R}^s} d^s \overline{\tau} \prod_{i=1}^s \mathbbm{I}_{A_{p_i,m}} ([t_i(\overline{\tau})])
\prod_{i \in J}
d_i \mathcal{K}^{(i)}_{m,\overline{\eps}, \overline{\delta}, \overline{k}, \overline{q}}(t_i(\overline{\tau})) ,t_{i+1}(\overline{\tau})) 
\\
& \times \prod_{i \not \in J}
d_i \mathcal{K}^{(i)}_{m,\overline{\eps}, \overline{\delta}, \overline{k}, \overline{q}}(t_i(\overline{\tau}) ,t_{i+1}(\overline{\tau})).
\end{split}
\end{equation}
Note that $[t_i (\overline{\tau})] \in A_{p_i,m}$ means that  
\begin{equation}
[\alpha_{p_i}^l \lambda_1 (2m)^{1/3} ] \leq t_i (\tau) \leq [\alpha_{p_i}^r \lambda_1 (2m)^{1/3} ].
\end{equation}
Dropping the integer parts gives a negligible error and this is equivalent to $t_i(\tau)/\lambda_1(2m)^{1/3} \in A_{p_i}$, where $A_{p_i}=[\alpha_{p_i}^l,\alpha_{p_i}^r]$.  By statement (3) in Proposition~\ref{Asymptotics}, for $i \in J'$ and $|\tau_i| \leq c_2 (\log m)^{1/6}$, 
\begin{equation}
d_i \mathcal{K}^{(i)}_{m,\overline{\eps}, \overline{\delta}, \overline{k}, \overline{q}}(t_i(\overline{\tau}) ,t_{i+1}(\overline{\tau})) = \frac{1}{\sqrt{4 \pi }} e^{-\frac{\tau_i^2}{4}} (1+o(1)) \mathbbm{I}_{k_i > k_{i+1}}
\end{equation}
as $m \to \infty$. Write
\begin{equation} \label{B0eqn}
B_0(\beta,\xi;\beta',\xi')= \phi_{\beta,\beta'} (\xi,\xi')
\end{equation}
and
\begin{equation}\label{B1eqn}
B_1(\beta,\xi;\beta',\xi')= \tilde{\mathcal{A}} (\beta, \xi;\beta', \xi').
\end{equation}
Then, for $i \in J$ and $|t_i (\overline{\tau})|$, $|t_{i+1} (\overline{\tau})| \leq C m^{1/3}$, 
\begin{equation}
d_i\mathcal{K}^{(i)}_{m,\overline{\eps}, \overline{\delta}, \overline{k},\overline{q}} (t_i (\overline{\tau}), t_{i+1} (\overline{\tau})) 
= B_{\delta_i} \left(\beta_{q_i} , \frac{t_i (\overline{\tau})}{\lambda_1 (2m)^{1/3}} ; \beta_{q_{i+1}} , \frac{t_{i+1} (\overline{\tau})}{\lambda_1 (2m)^{1/3}}\right) (1+o(1))  
\end{equation}
as $m \to \infty$.   Note that 
\begin{equation}
\lim_{m \to \infty} \frac{t_i(\overline{\tau})}{\lambda_1 (2m)^{1/3}} =  \lim_{m \to \infty} \sum_{j=i_1}^{i} \frac{d_j}{\lambda_1 (2m)^{1/3}}\tau_j= \sum_{\substack{j=i_1 \\ j \in J}}^i \tau_j
\end{equation}
for $i_1 \leq i <i_1 +s$.

It follows from the above asymptotic formulas and the estimates in Proposition~\ref{Asymptotics} that
\begin{equation}
\begin{split} \label{Slimit}
&\lim_{m\to \infty} S_s (\overline{\eps}, \overline{\delta}, \overline{k},\overline{p}, \overline{q})  \\
&= \int_{\mathbb{R}^r} \prod_{j \in J} d \tau_j 
\prod_{i=1}^s \mathbbm{I}_{A_{p_i}}  \left( \sum_{{j=i_1 , j \in J}}^i \tau_j
\right) \prod_{i \in J}  B_{\delta_i } \left( \beta_{q_i} , \sum_{{j=i_1 , j \in J}}^i \tau_j ; \beta_{q_{i+1}}, \sum_{{j=i_1 , j \in J}}^{i+1} \tau_j \right) \\
&\, \times
\prod_{i \in J'} \mathbbm{I}_{k_i>k_{i+1}} \frac{1}{\sqrt{4 \pi }} \int_{\mathbb{R}}d \tau_i  e^{-\frac{\tau_i^2}{4} }  \\
&= \prod_{i \in J'} \mathbbm{I}_{k_i>k_{i+1}}  \int_{\mathbb{R}^r} \prod_{j \in J} d \tau_j 
\prod_{i=1}^s \mathbbm{I}_{A_{p_i}}  \left( \sum_{{j=i_1 , j \in J}}^i \tau_j
\right) \prod_{i \in J}  B_{\delta_i } \left( \beta_{q_i} , \sum_{{j=i_1 , j \in J}}^i \tau_j ; \beta_{q_{i+1}}, \sum_{{j=i_1 , j \in J}}^{i+1} \tau_j \right).
\end{split}
\end{equation}
Note that a non-zero right side in~\eqref{Slimit} requires $p_i = p_{j_\alpha}$ for $j_{\alpha} \leq i <j_{\alpha+1}$ since otherwise 
\begin{equation}
\prod_{i=1}^s \mathbbm{I}_{A_{p_i}}  \left( \sum_{{j=i_1 , j \in J}}^i \tau_j 
\right)  =0.
\end{equation}
By the definition of $j_{\alpha}$, we have that $q_i=q_{j_{\alpha}}$ for $j_{\alpha} \leq i <j_{\alpha+1}$.  Note that the limit in~\eqref{Slimit} does not depend on $\eps$.  We have, for fixed $\delta$, $q$, which determine $J$ and $J'$, that
\begin{equation}
\lim_{m \to \infty }  \frac{1}{M^s} \sum_{\overline{k} \in [M]^s} \prod_{i \in J'} \mathbbm{I}_{k_i > k_{i+1}} = \frac{1}{\ell_1! \dots \ell_r!}.
\end{equation}
Thus, after an analogous change of variables to~\eqref{ttauchange}, we get
\begin{equation}
\begin{split} \label{Slimit2}
&\lim_{m \to \infty} \frac{1}{M^s} \sum_{(\overline{\delta}, \overline{k},\overline{p}, \overline{q})\in D_{s,3}^*} \prod_{i=1}^s (-1)^{\delta_i} w_{p_i,q_i} S_s( \overline{\eps}, \overline{\delta}, \overline{k},\overline{p}, \overline{q})  \\
&=\sum_{r=1}^s \sum_{\substack{\ell_1+\dots+\ell_r=s \\ \ell_1,\dots,\ell_r\geq 1}} \frac{\ell_1(-1)^r}{\ell_1!\dots \ell_r!} \sum_{\overline{p} \in [L_2]^r} \sum_{\overline{q} \in [L_1]^r} \sum_{\overline{\delta} \in \{0,1\}^r} \prod_{i=1}^r (-1)^{1+\delta_i} w_{p_i,q_i}^{\ell_i}  \\
&\times \int_{\mathbb{R}^r} d^r \overline{t} \prod_{i=1}^r \mathbbm{I}_{A_{p_i}} (t_i) B_{\delta_i} \left( \beta_{q_i} ,t_i;\beta_{q_{i+1}},t_{i+1} \right) 
\end{split}
\end{equation}
where the $\ell_1$ factor comes from the $\ell_1$ possible choices of $j_1$ as discussed above. By symmetry, we see that we can replace 
\begin{equation}
\sum_{\substack{\ell_1+\dots+\ell_r=s \\ \ell_1,\dots,\ell_r\geq 1}} \frac{\ell_1(-1)^r}{\ell_1!\dots \ell_r!} 
\end{equation}
in the right side of~\eqref{Slimit2} by
\begin{equation}
\frac{1}{r} \sum_{\substack{\ell_1+\dots+\ell_r=s \\ \ell_1,\dots,\ell_r\geq 1}} \frac{(\ell_1+ \dots +\ell_r)(-1)^r}{\ell_1! \dots \ell_r!} =  \frac{(-1)^r}{r} \sum_{\substack{\ell_1+\dots+\ell_r=s \\ \ell_1,\dots,\ell_r\geq 1}} \frac{s}{\ell_1!\dots \ell_r!}. 
\end{equation}
Thus, we find that
\begin{equation}
\begin{split}
&\lim_{m \to \infty} \frac{1}{M^s} \sum_{(\overline{\delta}, \overline{k},\overline{p}, \overline{q})\in D_{s,3}^*} \prod_{i=1}^s (-1)^{\delta_i} w_{p_i,q_i} S_s( \overline{\eps}, \overline{\delta}, \overline{k},\overline{p}, \overline{q})  \\
&=\sum_{r=1}^s \frac{(-1)^r}{r}\sum_{\substack{\ell_1+\dots+\ell_r=s \\ \ell_1,\dots,\ell_r\geq 1}}\frac{s}{\ell_1!\dots \ell_r!}  \sum_{\overline{p} \in [L_2]^r} \sum_{\overline{q} \in [L_1]^r}  \prod_{i=1}^r  w_{p_i,q_i}^{\ell_i}  \\ 
&\times \int_{\mathbb{R}^r} d^r \overline{t} \prod_{i=1}^r \mathbbm{I}_{A_{p_i}} (t_i) \mathcal{A} \left( \beta_{q_i} ,t_i;\beta_{q_{i+1}},t_{i+1} \right),
\end{split}
\end{equation}
since $\mathcal{A}= -B_0+B_1$ from~\eqref{eq:extendedAiry},~\eqref{B0eqn} and~\eqref{B1eqn}.  In order to get the limit of $U_3^*(m)$ in~\eqref{Ustar3m}, we need the following lemma, which we will prove in Section~\ref{sec:U0proof}.

\begin{lemma} \label{Tepsilon}
We have that
\begin{equation}
\sum_{\overline{\eps} \in \{0,1\}^s} Q( \overline{\eps}) =(-1)^s.
\end{equation}
\end{lemma}
Thus, using the estimate in Lemma~\ref{Srbound}, we see that, provided $|w_{p,q}| \leq R$ with $R$ sufficiently small, we can take the limit in~\eqref{Ustar3m} and get
\begin{equation}
\begin{split}
&\lim_{m \to \infty} U^*_3(m) =\sum_{s=1}^{\infty} \sum_{r=1}^s \frac{(-1)^{r+1}}{r} 
\sum_{\substack{\ell_1+\dots+\ell_r=s \\ \ell_1,\dots,\ell_r\geq 1}}\frac{1}{\ell_1!\dots \ell_r!}  \sum_{\overline{p} \in [L_2]^r} \sum_{\overline{q} \in [L_1]^r}  \prod_{i=1}^r  w_{p_i,q_i}^{\ell_i}  \\
& \times  \int_{\mathbb{R}^r} d^r \overline{t} \prod_{i=1}^r \mathbbm{I}_{A_{p_i}} (t_i) \mathcal{A} \left( \beta_{q_i} ,t_i;\beta_{q_{i+1}},t_{i+1} \right)  \\
&= \log \det \left( \mathbbm{I} +(e^{\Psi}-1 ) \right)_{L^2(\{\beta_1,\dots,\beta_{L_1}\} \times \mathbb{R})},
\end{split}
\end{equation}
where
$\Psi(x)=\sum_{p=1}^{L_2} \sum_{q=1}^{L_1} w_{p,q} \mathbbm{I}_{\{\beta_q\} \times A_p} (x)$ as defined in~\eqref{phi}
for  $x \in \{\beta_1,\dots,\beta_q \} \times \mathbb{R}$.  This completes the proof of the theorem.

\section{Proofs of Lemmas~\ref{U0} and~\ref{Tepsilon}} \label{sec:U0proof}

In this section, we will give the proof of Lemma~\ref{U0} followed by the proof of Lemma~\ref{Tepsilon}.  These were both stated without proof in Section~\ref{sec:mainthmproof}.

Before giving the proof of Lemma~\ref{U0}, we recall notation and give
some  preliminaries. As  in Section~\ref{sec:mainthmproof},  we assume
that the notation is cyclic, that  is $z_{r+1}=z_1$ in all products of
size   $r$.  Note   that  since   $\mathtt{K}_{m,0}$  is   related  to
$\tilde{\mathcal{K}}_{m,0}= a  \mathrm{i} \mathbb{K}^{-1}_{1,1}$  by a
conjugation, see~\eqref{Ktilde},~\eqref{Ktilde2}  and~\eqref{KmKm}, we
have
\begin{equation}
\prod_{i=1}^r \mathtt{K}_{m,0} (z_i,z_{i+1}) = \prod_{i=1}^r a \mathrm{i} \mathbb{K}^{-1}_{1,1} (x(z_{i+1}),y(z_i)) \label{Kidentity}.
\end{equation}
Let $t=t(z)$, $t'=t(z')$, $\eps=\eps(z)$ and $\eps'=\eps(z')$, where $z,z' \in \mathcal{L}_m(q,k)$.  From~\eqref{gasphaseeqn}, we see that 
\begin{equation}
\mathbb{K}_{1,1}^{-1}(x(z'),y(z)) = - \frac{\mathrm{i}^{1+h(\eps_1,\eps_2)}}{(2\pi \mathrm{i})^2} \int_{\Gamma_1} \frac{d u_1}{u_1} \int_{\Gamma_1} \frac{du_2}{u_2} \frac{a^{\eps} u_2^{1-h(\eps,\eps')} +a^{1-\eps} u_1 u_2^{h(\eps,\eps')} }{\tilde{c}(u_1,u_2) u_1^{\frac{x_1(z')-y_1(z)+1}{2}} u_2^{\frac{x_2(z')-y_2(z)+1}{2}}}.
\end{equation} 
Now, we have $x_2(z')-y_2(z)=2(t'-t)-1+2\eps$ and $x_1(z')-y_1(z)=2(t'-t)+1-2\eps'$ by~\eqref{tz} and~\eqref{xyz}. Define
\begin{equation}
\mathcal{G}_{\eps,\eps'}(t)= \frac{a \mathrm{i}^{h(\eps,\eps')}}{(2\pi \mathrm{i})^2} \int_{\Gamma_1} \frac{du_1}{u_1}  \int_{\Gamma_1} \frac{du_2}{u_2}  \frac{a^{\eps}u_1^{-1+\eps+\eps'}  u_2^{1-h(\eps,\eps')} +a^{1-\eps} u_1^{\eps+\eps'}  u_2^{h(\eps,\eps')} }{\tilde{c}(u_1,u_2) (u_1 u_2)^{t +\eps}}.
 \label{Lepsilon}
\end{equation}
It follows that 
\begin{equation}
a \mathrm{i} \mathbb{K}_{1,1}^{-1} (x(z'),y(z)) = \mathcal{G}_{\eps,\eps'}(t'-t) 
\end{equation}
and consequently 
\begin{equation}
\prod_{i=1}^r \mathtt{K}_{m,0} (z_i,z_{i+1}) = \prod_{i=1}^r  \mathcal{G}_{\eps_i,\eps_{i+1}}(t_{i+1}-t_i) \label{KLidentity}
\end{equation}
if $z_i \in \mathcal{L}_m(q,k)$, $\eps_i = \eps(z_i)$ and $t_i=t(z_i)$ for $1 \leq i \leq r$.  By making the change of variables $u_1=u$, $u_2=\omega/u$ in~\eqref{Lepsilon}, we obtain
\begin{equation}
\mathcal{G}_{\eps,\eps'}(t) = \frac{a \mathrm{i}^{h(\eps,\eps')}}{2 \pi \mathrm{i}} \int_{\Gamma_1} \frac{f_{\eps,\eps'}(\omega)}{\omega^t} \frac{d\omega }{\omega} \label{Lepsilon2}
\end{equation}
where 
\begin{equation} \label{fepseps'}
f_{\eps,\eps'}(\omega) = \frac{1}{2\pi \mathrm{i}} \int_{\Gamma_1} \frac{du}{u} \frac{a^{\eps } u^{-2(1-\eps)(1-\eps')} \omega^{1-\eps-h(\eps,\eps')} +a^{1-\eps} u^{2\eps \eps'} \omega^{h(\eps,\eps')-\eps}}{\tilde{c}(u,\omega/u)}.
\end{equation}
We have the following lemma
\begin{lemma} \label{lem:feps}
Let $f_{\eps,\eps'}(\omega)$ be defined in~\eqref{fepseps'}. Then, we have the relations,
\begin{equation}
f_{0,0}(\omega)=f_{1,1}(\omega)
\end{equation}
and 
\begin{equation}
a f_{0,0}(\omega)-a^2( f_{0,0} (\omega)^2 +f_{0,1}(\omega)f_{1,0}(\omega)) = 0.
\end{equation}
\end{lemma} 
\begin{proof}
From~\eqref{ctilde}, we have
\begin{equation}  
\begin{split}
\tilde{c} (\sqrt{u},\omega/\sqrt{u}) & =\frac{a}{u \omega} (u^2 + (1+2(a+1/a)\omega+\omega^2)u+\omega^2) \\
&=\frac{a}{u\omega} (u-r_1(\omega))(u-r_2(\omega)).  
\end{split}
\end{equation} 
The term in the parenthesis on the right side of the first line of the above equation is a quadratic in $u$ and the second line gives the factorization into two roots,  $r_1(\omega)$ and $r_2(\omega)$.   We have that $r_1(\omega)r_2(\omega)=\omega^2$ and so for $\omega \in \mathbb{T}$, we choose $|r_1(\omega)|<1$ and $|r_2(\omega)|>1$. 

Making the change of variables $u \mapsto \sqrt{u}$ for $f_{\eps,\eps'}(\omega)$, defined in~\eqref{fepseps'}, gives
\begin{equation}
\begin{split}
f_{\eps,\eps'} (z) &= \frac{1}{2 \pi \mathrm{i}} \int_{\Gamma_1} \frac{du}{u}
\frac{a^{\eps} \omega^{1-\eps- h(\eps,\eps')} u^{-(1-\eps)(1-\eps')}+a^{1-\eps} u^{\eps \eps'} \omega^{h(\eps,\eps')-\eps}
}{\tilde{c}(\sqrt{u},\omega/\sqrt{u})}\\
&= \frac{1}{2 \pi \mathrm{i}} \int_{\Gamma_1} du \,
\frac{a^{\eps-1} \omega^{2-h(\eps,\eps')-\eps} u^{-(1-\eps)(1-\eps')}+a^{-\eps} u^{\eps \eps'} \omega^{1+h(\eps,\eps')-\eps}}{(u-r_1(\omega))(u-r_2(\omega))}.
\end{split}
\end{equation} 
 In the above integrand for $(\eps,\eps')\not=(0,0)$, then $(1-\eps)(1-\eps')=0$ which means that there is only residue at $u=r_1(z)$. This is easily computed and gives
\begin{equation}
\frac{a^{\eps-1 }\omega^{2-h(\eps,\eps')-\eps} +a^{-\eps} r_1(\omega)^{\eps \eps'} \omega^{1+h(\eps,\eps') -\eps}}{r_1(\omega)-r_2(\omega)}.
\end{equation}
For $(\eps,\eps')=(0,0)$, there are residues at $u=r_1(\omega)$ and $u=0$ which give
\begin{equation}
\begin{split}
\frac{a^{-1} \omega^{2} r_1(\omega)^{-1}+ \omega}{r_1(\omega)-r_2(\omega)}+\frac{a^{-1} \omega^{2}}{r_1(\omega)r_2(\omega)}&= \frac{r_1(\omega)r_2(\omega)(a^{-1} \omega^{2} r_1(\omega)^{-1}+ \omega)+a^{-1} \omega^{2}(r_1(\omega)-r_2(\omega))}{r_1(\omega)r_2(\omega)(r_1(\omega)-r_2(\omega))}\\
&=\frac{a^{-1}r_1(\omega) +\omega}{r_1(\omega)-r_2(\omega)}
\end{split}
\end{equation}
where we have used $r_1(\omega)r_2(\omega)=\omega^2$. We have arrived at
\begin{equation} \label{fepseps'form}
f_{\eps,\eps'} (\omega)=\frac{a^{\eps-1} r_1(\omega)^{(1-\eps)(1- \eps')} \omega^{2-2(1-\eps)(1-\eps')-h(\eps,\eps')-\eps} +a^{-\eps} \omega^{1+h(\eps,\eps')-\eps}r_1(\omega)^{\eps \eps'} }{r_1(\omega) - r_2(\omega)} 
\end{equation} 
Using the above equation, the first equation in Lemma~\ref{lem:feps} immediately follows.  For the second in equation in Lemma~\ref{lem:feps}, using~\eqref{fepseps'form} we have
\begin{equation}
\begin{split}
&a f_{0,0}(\omega) - a^2\left( f_{0,0} (\omega)^2 +{f_{0,1}(\omega) f_{1,0}(\omega)} \right)\\
&= a \frac{a^{-1} r_1(\omega)+\omega}{r_1(\omega)-r_2(\omega)} -\frac{a^2}{(r_1(\omega)-r_2(\omega))^2} \left((a^{-1} r_1(\omega)+\omega)^2+ (a^{-1} \omega +\omega^2)(1+a^{-1} \omega)\right)\\
&=-\frac{2 a^2 \omega^2+a r_1(\omega) \omega+a r_2(\omega) \omega+a \omega^3+a \omega+r_1(\omega) r_2(\omega)+\omega^2}{(r_1(\omega)-r_2(\omega))^2}\\
&=-\frac{\omega \left(2 a^2 \omega+a \left(r_1(\omega)+r_2(\omega)+\omega^2+1\right)+2 \omega\right)}{(r_1(\omega)-r_2(\omega))^2}=0
\end{split}
\end{equation}
where we have used $r_1(\omega)r_2(\omega)=\omega^2$ and $r_1(\omega)+r_2(\omega)=-(1+2(a+1/a)\omega+\omega^2)$ as required.

\end{proof}

The next lemma expresses the exponential decay of correlation in a pure gas phase.
\begin{lemma} \label{Lepsilonbound}
There are constants $C,c_1>0$ so that
\begin{equation}
| \mathcal{G}_{\eps,\eps'}(t) | \leq C e^{-c_1|t|}
\end{equation}
for all $t \in \mathbb{Z}$ and $\eps,\eps' \in \{0,1\}$. 
\end{lemma}
\begin{proof}
We see from the proof of the previous lemma that $f_{\eps,\eps'}(\omega)$ is an analytic function in the neighborhood of the unit circle.  Let $t>0$ and take $r>1$, but close to $1$ so that $f_{\eps,\eps'}(\omega)$ is analytic in $\{\omega: 1\leq |\omega| \leq r\}$.  We see from~\eqref{Lepsilon2} and Cauchy's theorem that
\begin{equation}
|\mathcal{G}_{\eps,\eps'}(t)| =\frac{a}{2\pi} \left| \int_{\Gamma_r} \frac{f_{\eps,\eps'}(\omega)}{\omega^t} \frac{d\omega}{\omega} \right| \leq \frac{C}{r^t}.
\end{equation}
If $t<0$, we take $r<1$ instead.
\end{proof}

We are now ready to prove Lemma~\ref{U0}. 
\begin{proof}[Proof of Lemma~\ref{U0}] 

Let $(\overline{\delta} ,\overline{k},\overline{p},\overline{q}) \in D_{r,0}$ so that $\delta_i =0$, $p_i=p$, $q_i=q$, $k_i=k$, $1 \leq i\leq r$.  Thus,
\begin{equation} 
T_0 (m,r,\overline{\ell}) =\sum_{\overline{\eps} \in \{0,1\}^r} \prod_{i=1}^r (-1)^{\ell_i \eps_i} \sum_{k=1}^M \sum_{p=1}^{L_2} \sum_{q=1}^{L_1} w_{p,q}^{\ell_1+\dots+\ell_r} \sum_{\overline{z} \in(\mathcal{L}_m)^r} \prod_{i=1}^r \mathbbm{I}_{p,q,k}^{\eps_i} (z_i) a \mathrm{i} \mathbb{K}_{1,1}^{-1} (x(z_{i+1}) , y(z_i))
\end{equation}
by~\eqref{Tjmrell} and~\eqref{Kidentity}.  Recalling the definition of $A_{p,m}$ in~\eqref{Apm} and using~\eqref{KLidentity}, we have that
\begin{equation}
 \sum_{\overline{z} \in(\mathcal{L}_m)^r} \prod_{i=1}^r \mathbbm{I}_{p,q,k}^{\eps_i} (z_i) a \mathrm{i} \mathbb{K}_{1,1}^{-1} (x(z_{i+1}) , y(z_i)) = \sum_{\overline{t} \in \mathbb{Z}^r}  \prod_{i=1}^r \mathbbm{I}_{A_{p,m}} (t_i)  \mathcal{G}_{\eps_i,\eps_{i+1}} (t_{i+1}-t_i). 
\end{equation}
With the above equations and~\eqref{Ujm}, we obtain
\begin{equation} \label{Uzero}
\begin{split}
U_0(m)&=M \sum_{p=1}^{L_2} \sum_{q=1}^{L_1} \Bigg( \sum_{s=1}^{\infty} \frac{w_{p,q}^s}{M^s} \sum_{r=1}^s \frac{(-1)^{r+1}}{r}
\\
 &\times \sum_{\substack{\ell_1+\dots+\ell_r=s \\ \ell_1,\dots, \ell_r \geq 1}}  \frac{1}{\ell_1! \dots \ell_r!}
\sum_{\overline{\eps} \in \{0,1\}^r} \prod_{i=1}^r (-1)^{\ell_i \eps_i} \sum_{\overline{t} \in \mathbb{Z}^r} \prod_{i=1}^r \mathbbm{I}_{A_p,m} (t_i) \mathcal{G}_{\eps_i,\eps_{i+1}} (t_{i+1}-t_i) \Bigg)
\end{split}
\end{equation}
The result of the lemma now follows from~\eqref{Uzero} and the next claim, since we get the estimate
\begin{equation}
|U_0(m) | \leq \frac{C}{M}.
\end{equation}
\begin{claim}\label{claim1}
There is a constant $C$ and an $R>0$ so that
\begin{equation}
\begin{split} \label{Mtwoest}
\Bigg| \sum_{s=1}^{\infty} \frac{w^s}{M^s}\sum_{r=1}^s \frac{(-1)^{r+1}}{r}
  \sum_{\substack{\ell_1+\dots+\ell_r=s \\ \ell_1,\dots, \ell_r \geq 1}}  
\frac{1}{\ell_1! \dots \ell_r!}
\sum_{\overline{\eps} \in \{0,1\}^r} 
\sum_{\overline{t} \in \mathbb{Z}^r} \prod_{i=1}^r(-1)^{\ell_i \eps_i}\mathbbm{I}_{A_p,m} (t_i) \mathcal{G}_{\eps_i,\eps_{i+1}} (t_{i+1}-t_i) \Bigg| \leq \frac{C}{M^2} 
\end{split}
\end{equation}
for $|w| \leq R$ and $1 \leq p \leq L_2$.
\end{claim}

\begin{proof}[Proof of Claim~\ref{claim1}]
From~\eqref{Lepsilon2}, we see that 
\begin{equation} \label{Fourier}
\mathcal{G}_{\eps,\eps'}(t) = a \mathrm{i}^{h(\eps,\eps')} \widehat{f}_{\eps,\eps'}(t)
\end{equation}
for $t \in \mathbb{Z}$ where $\widehat{f}_{\eps,\eps'}(t)$ is the $t^{\mathrm{th}}$ Fourier coefficient of $f_{\eps,\eps'}$.  Thus, we have 
\begin{equation}
\begin{split}
\sum_{\overline{t} \in \mathbb{Z}^r} \prod_{i=1}^r\mathbbm{I}_{A_p,m} (t_i) \mathcal{G}_{\eps_i,\eps_{i+1}} (t_{i+1}-t_i)
=\prod_{i=1}^r a \mathrm{i}^{h(\eps_i,\eps_{i+1})} \sum_{\overline{t} \in \mathbb{Z}^r} \prod_{i=1}^r\mathbbm{I}_{A_p,m} (t_i) \widehat{f}_{\eps_i,\eps_{i+1}} (t_{i+1}-t_i).
\end{split}
\end{equation}
Using properties of convolutions of Fourier coefficients, we have
\begin{equation}
\begin{split}
\sum_{t_2,\dots,t_r \in \mathbb{Z}} \prod_{i=1}^r \widehat{f}_{\eps_i,\eps_{i+1}} (t_{i+1}-t_i)& = 
\sum_{t_r \in \mathbb{Z}} (\widehat{f_{\eps_1,\eps_2} \dots  f_{\eps_{r-1},\eps_r}} )(t_r-t_1) \widehat{f}_{\eps_r,\eps_1}(t_1-t_r)  \\
&=\sum_{t_r \in \mathbb{Z}} (\widehat{f_{\eps_1,\eps_2} \dots  f_{\eps_{r-1},\eps_r}} )(t_r) \widehat{f}_{\eps_r,\eps_1}(-t_r)  \\
&=( \widehat{f_{\eps_1,\eps_2} \dots  f_{\eps_{r-1},\eps_r}} )(0)= \frac{1}{2\pi \mathrm{i}} \int_{\Gamma_1} \frac{d\omega}{\omega} \prod_{i=1}^r f_{\eps_i,\eps_{i+1}}(\omega)
\end{split}
\end{equation}
for $r \geq 2$.   Thus, for $r \geq 2$ we have
\begin{equation}
\begin{split} \label{Traceest}
&\left| \sum_{\overline{t} \in\mathbb{Z}^r} \prod_{i=1}^r \mathbbm{I}_{A_{p,m}} (t_i) \widehat{f}_{\eps_i,\eps_{i +1} } (t_{i+1}-t_i ) - \frac{|A_{p,m}|}{2\pi \mathrm{i} } \int_{\Gamma_1} \frac{d\omega}{\omega}  \prod_{i=1}^r f_{\eps_i,\eps_{i+1}} (\omega) \right|  \\
&=\left| \sum_{\overline{t} \in\mathbb{Z}^r} \mathbbm{I}_{A_{p,m}} (t_1) \left( \prod_{i=2}^r \mathbbm{I}_{A_{p,m}} (t_i)-1 \right)\prod_{i=1}^r  \widehat{f}_{\eps_i,\eps_{i +1} } (t_{i+1}-t_i )\right|  \\
&\leq  \sum_{\overline{t} \in\mathbb{Z}^r} \mathbbm{I}_{A_{p,m}} (t_1) \left( \sum_{j=2}^r \mathbbm{I}_{A^c_{p,m}} (t_j) \right) \prod_{i=1}^r Ce^{-c_1 |t_{i+1}-t_i|} 
\end{split}
\end{equation}
by Lemma~\ref{Lepsilonbound} and Eq.~\eqref{Fourier}. 

Introduce new coordinates $s_1=t_1$, $s_{i} = t_{i}-t_{i-1}$, $2 \leq i \leq r$.  The inverse is 
\begin{equation}
t_j =\sum_{i=1}^j s_i
\end{equation}
so we get a bijection from $\mathbb{Z}^r$ to $\mathbb{Z}^r$. We see that the right side in~\eqref{Traceest} is less than or equal to
\begin{equation}
\begin{split}
& C^r\sum_{\overline{s} \in\mathbb{Z}^r} \mathbbm{I}_{A_{p,m}} (s_1) \left( \sum_{j=2}^r \mathbbm{I}_{A^c_{p,m}} (s_1 +\dots +s_j ) \right) \prod_{i=1}^r Ce^{-c_1 \sum_{i=2}^r|s_i| -c_1 |s_2+\dots+s_r|}  \\
&\leq C^r \sum_{j=2}^r \sum_{\sigma, {s_1} \in\mathbb{Z}} \sum_{\substack{s_2,\dots, s_r \in \mathbb{Z} \\ s_2+\dots+s_j=\sigma}} \mathbbm{I}_{A_{p,m}} (s_1)  \mathbbm{I}_{A^c_{p,m}} (s_1 +\sigma)  e^{-\frac{c_1}{2}|\sigma| -\frac{c_1}{2} \sum_{i=2}^r |s_i|}  \\
& \leq C^r \sum_{j=2}^r \sum_{\sigma, {s_1} \in\mathbb{Z}} \mathbbm{I}_{A_{p,m}} (s_1)\mathbbm{I}_{A^c_{p,m}} (s_1 +\sigma) e^{-\frac{c_1}{2} |\sigma|} \left( \sum_{s_2,\dots,s_r \in \mathbb{Z}} e^{-\frac{c_1}{2} (|s_2|+\dots+|s_r|)} \right) \\
& \leq C^r \sum_{\sigma, {s_1} \in\mathbb{Z}} \mathbbm{I}_{A_{p,m}} (s_1)\mathbbm{I}_{A^c_{p,m}} (s_1 +\sigma) e^{-\frac{c_1}{2} |\sigma|} \leq C^r.
\end{split}
\end{equation}
Thus, we find
\begin{equation} \label{Traceest2}
\left| \sum_{\overline{t} \in\mathbb{Z}^r} \prod_{i=1}^r \mathbbm{I}_{A_{p,m}} (t_i) \widehat{f}_{\eps_i,\eps_{i +1} } (t_{i+1}-t_i ) - \frac{|A_{p,m}|}{2\pi \mathrm{i} } \int_{\Gamma_1} \frac{d\omega}{\omega}  \prod_{i=1}^r f_{\eps_i,\eps_{i+1}} (\omega) \right|  \leq C^r.
\end{equation}
Write, 
\begin{equation}
\begin{split}
\Sigma_1 &= \sum_{s=1}^\infty \frac{w^s}{M^s} \frac{1}{s!} \sum_{\eps_1 \in \{0,1\}} \sum_{t_1 \in \mathbb{Z}} (-1)^{s \eps_1} \mathbbm{I}_{A_{p,m}} (t_1) a \mathrm{i}^{h(\eps_1,\eps_1)}  \widehat{f}_{\eps_1,\eps_1}(0) \\
& =
\sum_{s=1}^\infty \frac{w^s}{M^s} \frac{1}{s!} \sum_{\eps_1 \in \{0,1\}}(-1)^{s \eps_1}  \frac{|A_{p,m}|}{2 \pi \mathrm{i}} \int_{\Gamma_1} \frac{d \omega}{\omega} a \mathrm{i}^{h(\eps_1,\eps_1)} f_{\eps_1,\eps_1}(\omega),
\end{split}
\end{equation}
and
\begin{equation}
\Sigma_2 = \sum_{s=2}^\infty \frac{w^s}{M^s} \sum_{r=2}^{s} \frac{(-1)^{r+1}}{r} \sum_{\substack{\ell_1+\dots +\ell_r=s \\ \ell_1,\dots,\ell_r \geq 1}} \frac{1}{\ell_1! \dots \ell_r!} \sum_{\overline{\eps} \in \{0,1\}^r } \sum_{\overline{t} \in \mathbb{Z}^r} 
\prod_{i=1}^r (-1)^{\ell_i \eps_i} \mathbbm{I}_{A_{p,m}} (t_i) a \mathrm{i}^{h(\eps_i,\eps_{i+1})} \hat{f}_{\eps_i,\eps_{i+1}}(t_{i+1} -t_i),
\end{equation}
so that the left side of~\eqref{Mtwoest} is $|\Sigma_1+\Sigma_2|$.  Now, using~\eqref{Traceest2}, we see that
\begin{equation}
\begin{split}
&\left| \Sigma_2 - \sum_{s=2}^\infty \frac{w^s}{M^s} \sum_{r=2}^{s} \frac{(-1)^{r+1}}{r} \sum_{\substack{\ell_1+\dots +\ell_r=s \\ \ell_1,\dots,\ell_r \geq 1}} \frac{1}{\ell_1! \dots \ell_r!} \sum_{\overline{\eps} \in \{0,1\}^r } \frac{|A_{p,m}|}{2\pi \mathrm{i}} \int_{\Gamma_1}\frac{d\omega}{\omega} \prod_{i=1}^r (-1)^{\ell_i \eps_i} a \mathrm{i}^{h(\eps_i,\eps_{i+1})} f_{\eps_i,\eps_{i+1}} (\omega)  \right| \\
& \leq \sum_{s=2}^\infty \frac{|w|^s}{M^s} \sum_{r=2}^{s}  \frac{1}{r} \sum_{\substack{\ell_1+\dots +\ell_r=s \\ \ell_1,\dots,\ell_r \geq 1}} \frac{C^r}{\ell_1! \dots \ell_r!} \leq \sum_{s=2}^{\infty} \frac{R^s C^s}{M^s} \leq \frac{C}{M^2}
\end{split}
\end{equation}
if $R$ is sufficiently small.  Thus,
\begin{equation} \label{Sigmasest}
\begin{split}
&| \Sigma_1+\Sigma_2| \leq\\
&  \left| \sum_{s=1}^{\infty} \frac{w^s}{M^s} \sum_{r=1}^s \frac{(-1)^{r+1}}{r}  \sum_{\substack{\ell_1+\dots +\ell_r=s \\ \ell_1,\dots,\ell_r \geq 1}} \frac{|A_{p,m}|}{\ell_1! \dots \ell_r!}
\sum_{\overline{\eps} \in \{0,1\}^r } \frac{1}{2\pi \mathrm{i}} \int_{\Gamma_1} \frac{d\omega}{\omega} \prod_{i=1}^r  (-1)^{\ell_i \eps_i} a \mathrm{i}^{h(\eps_i,\eps_{i+1})} f_{\eps_i,\eps_{i+1}} (\omega)  \right| +\frac{C}{M^2}.
\end{split}
\end{equation}
Let $F_{\omega}=(F_{\omega}(\eps_1,\eps_2) )_{0 \leq \eps_1,\eps_2 \leq 1}$ be the two by two matrix with elements $F_{\omega}(\eps_1,\eps_2) ) = a \mathrm{i}^{h(\eps_1,\eps_2)} f_{\eps_1,\eps_2}(\omega)$ for $0\leq \eps_1,\eps_2 \leq 1$, and let $\eta(\eps_1)=(-1)^{\eps_1}$.  Then, the expression between the absolute value signs in the right side of~\eqref{Sigmasest} can be written as
\begin{equation} \label{lastcumulantexpand}
\frac{|A_{p,m}|}{2 \pi \mathrm{i}} \int_{\Gamma_1} \frac{d\omega}{\omega} \sum_{s=1}^\infty \frac{w^s}{M^s} \sum_{r=1}^s \frac{(-1)^{r+1}}{r} \sum_{\substack{\ell_1+\dots +\ell_r=s \\ \ell_1,\dots,\ell_r \geq 1}} \frac{1}{\ell_1! \dots \ell_r!} \mathrm{tr} (\eta^{\ell_1} F_\omega \dots \eta^{\ell_r} F_\omega ).
\end{equation}
Here, we view $F_\omega$ as an operator with kernel $F_\omega$ and on functions $\{0,1\} \to \mathbb{C}$, that is, the trace is for a product of two $2 \times 2$ matrices.   The expression in the integrand above is a cumulant expansion of $\log \det ( \mathbbm{I} + (e^{\frac{\omega}{M} \eta}-1)F_\omega )$.  This means that~\eqref{lastcumulantexpand} equals
\begin{equation}
\frac{|A_{p,m}|}{2\pi \mathrm{i}} \int_{\Gamma_1} \frac{d\omega}{\omega} \log \det \left( \mathbbm{I} + (e^{\frac{\omega}{M} \eta}-1)F_\omega \right)
\end{equation}
provided that $R$ is small enough.   
The above determinant can be written explicitly and is given by
\begin{equation}
\begin{split}
&\det\left(
 \left( \begin{array}{cc}
1 & 0 \\
0 & 1 \end{array} \right) 
+
\left( \begin{array}{cc}
(e^{\frac{w}{M}}-1) a f_{0,0}(\omega) & (e^{\frac{w}{M}} -1) a\mathrm{i} f_{0,1}(\omega) \\
(e^{-\frac{w}{M}}-1) a \mathrm{i}{f_{1,0}(\omega)} & (e^{-\frac{w}{M}} -1) a {f_{1,1}(\omega)} \end{array} \right) \right)\\&=
(1+(e^{\frac{w}{M}}-1)af_{0,0}(\omega) )(1+(e^{-\frac{w}{M}}-1)af_{1,1}(\omega)  + (e^{\frac{w}{M}}-1)(e^{-\frac{w}{M}}-1)a^2 f_{0,1}(\omega) f_{1,0}(\omega)\\
&=1+(e^{\frac{w}{M}}-1)af_{0,0}(\omega)+(e^{-\frac{w}{M}}-1)af_{1,1}(\omega) +a^2(2-e^{\frac{w}{M}}-e^{-\frac{w}{M}})f_{0,0}(\omega)f_{1,1}(\omega) \\&+
a^2(2-e^{\frac{w}{M}}-e^{-\frac{w}{M}})f_{0,1}(\omega) f_{1,0}(\omega)\\
&=1+(e^{\frac{w}{M}}-1)af_{0,0}(\omega)+(e^{-\frac{w}{M}}-1)af_{0,0}(\omega) +a^2(2-e^{\frac{w}{M}}-e^{-\frac{w}{M}})f_{0,0}(\omega)^2 \\&+
a^2(2-e^{\frac{w}{M}}-e^{-\frac{w}{M}})f_{0,1}(\omega) f_{1,0}(\omega)\\
	&=1+(e^{\frac{w}{M}}-1)af_{0,0}(\omega)+(e^{-\frac{w}{M}}-1)af_{0,0}(\omega) +a(2-e^{\frac{w}{M}}-e^{-\frac{w}{M}})f_{0,0}(\omega), 
\end{split}
\end{equation}
where the third equality follows from the first relation in Lemma~\ref{lem:feps} and the fourth equality follows from the second relation in Lemma~\ref{lem:feps}. We conclude that
\begin{equation}
\det \left( \mathbbm{I}+ (e^{\frac{w}{M} \eta}-1 ) F_\omega\right) =1
\end{equation}
and so we have shown that $|\Sigma_1+\Sigma_2| \leq C/M^2$.  This proves the claim.
\end{proof}
The proof of the claim concludes the proof of Lemma~\ref{U0}.

\end{proof}

We now give the proof of Lemma~\ref{Tepsilon}.

\begin{proof}[Proof of Lemma~\ref{Tepsilon}]

We have that from~\eqref{Pepsilonell} and~\eqref{PtoQ} 
\begin{equation}
Q(\overline{\eps})=P(\overline{\eps},\overline{1}) = \prod_{i=1}^s \frac{a \mathrm{i}}{2} \sqrt{1-2c} \mathtt{g}_{\eps_1,\eps_{i+1}} \mathcal{C}^{\eps_i+\eps_{i+1}-2}.
\end{equation}
From this we see that the left side of (\ref{Tepsilon}) is  the trace of the $s^{\mathrm{th}}$ power of a two by two matrix where the $(\eps_1+1,\eps_2+1)^{\mathrm{th}}$ entry is
\begin{equation}
 a \mathrm{i} \frac{\sqrt{1-2c}}{2}  \frac{\mathtt{g}_{\eps_1,\eps_2}}{\mathcal{C}^{2-\eps_1-\eps_2}}
\end{equation}
for $\eps_1,\eps_2 \in \{0,1\}$.  These entries are simplified using the expressions of $\mathtt{g}_{\eps_1,\eps_2}$ and $\mathcal{C}$ given above. Thus, the two by two matrix has  the explicit form
\begin{equation}
\left(
\begin{array}{cc}
 -\frac{1}{2} \left(1+\frac{1}{\sqrt{a^2+1}}\right) & \frac{a\mathrm{i}}{2 \sqrt{a^2+1}} \\
 -\frac{a\mathrm{i}}{2 \sqrt{a^2+1}} & -\frac{1}{2}+\frac{1}{2 \sqrt{a^2+1}} \\
\end{array}
\right)
\end{equation}
which has eigenvalues $0$ and $-1$, as required.

\end{proof}

\section{Proof of Proposition~\ref{Asymptotics}} \label{sec:proofasymptotics}

In this section, we give the proof of Proposition~\ref{Asymptotics}.  In order to give this proof, we rely on various results from~\cite{CJ:16} which are recalled below.

Let $\alpha_x,\alpha_y,\beta_x,\beta_y \in \mathbb{R}$, $k_x,k_y \in \mathbb{Z}$ and $f_x,f_y \in \mathbb{Z}^2$.  Set
\begin{equation}
\begin{split} \label{xyscaling}
x&=(\rho_m +2[\alpha_x \lambda_1 (2m)^{1/3}])e_1 -(2[\beta_x \lambda_2 (2m)^{2/3}+k_x \lambda_2 (\log m)^2])e_2 +f_x\\
y&=(\rho_m +2[\alpha_y \lambda_1 (2m)^{1/3}])e_1 -(2[\beta_y \lambda_2 (2m)^{2/3}+k_y \lambda_2 (\log m)^2])e_2 +f_y.
\end{split}
\end{equation}
From~\cite[Theorem 2.7]{CJ:16} and its proof, we have
\begin{thma}[\cite{CJ:16}]
\label{Airyasymptotics}
Assume that $x \in \mathtt{W}_{\eps_x}$ and $y \in \mathtt{B}_{\eps_y}$ are given by~\eqref{xyscaling} with $\eps_x,\eps_y \in\{0,1\}$.  Furthermore, assume that $|\alpha_x|,|\alpha_y|,|\beta_x|,|\beta_y|,|f_x|,|f_y|\leq C$ for some constant $C>0$ and that $|k_x|,|k_y|\leq M$.  Then, as $m \to \infty$
\begin{equation}
\begin{split}\label{As1}
\mathbb{K}_{\mathrm{A}}(x,y)&= \mathrm{i}^{y_1-x_1+1} \mathcal{C}^{\frac{-2-x_1+x_2+y_1-y_2}{2}} c_0 \mathtt{g}_{\eps_x,\eps_y} e^{\alpha_y \beta_y -  \alpha_x \beta_x -\frac{2}{3} (\beta_x^3-\beta_y^3)}
\\ &\times 
(2m)^{-\frac{1}{3}} (\tilde{\mathcal{A} }(\beta_x , \alpha_x+\beta_x^2; \beta_y,\alpha_y+\beta_y^2)+o(1)).
\end{split}
\end{equation}
Also, as $m\to \infty$,
\begin{equation}
\begin{split}\label{As2}
\mathbb{K}_{1,1}^{-1}(x,y)&= \mathrm{i}^{y_1-x_1+1} \mathcal{C}^{\frac{-2-x_1+x_2+y_1-y_2}{2}} c_0 \mathtt{g}_{\eps_x,\eps_y} e^{\alpha_y \beta_y -  \alpha_x \beta_x -\frac{2}{3} (\beta_x^3-\beta_y^3)}\\ &\times (2m)^{-\frac{1}{3}}  ( \phi_{\beta_x,\beta_y} ( \alpha_x+\beta_x^2 ;\alpha_y+\beta_y^2) +o(1)).
\end{split}
\end{equation}

\end{thma}
\begin{remark}
The difference between the above version of the theorem and the statement given in~\cite[Theorem 2.7]{CJ:16}, is that there is a positional change of the vertices $x$ and $y$ by at most $|k_x\lambda_2(\log m)^2|$ and $|k_y\lambda_2(\log m)^2|$ and the reverse time orientation, which simply consists of the change $\beta_x\mapsto -\beta_x$ and $\beta_y\mapsto -\beta_y$.  By comparing the statement of~\cite[Theorem 2.7]{CJ:16} and Theorem~\ref{Airyasymptotics}, the positional change affects the exponent of $\mathcal{C}$ and the error term, where we remind the reader that $|G(\mathrm{i})|$ in~\cite{CJ:16} is equal to $\mathcal{C}$ in this paper.

More explicitly, this positional change  only alters the Taylor series
computation                of                 the                ratio
$H_{x_1+1,x_2}(\omega_1)/  H_{y_1,y_2+1}(\omega_2)$  using  the  local
change      of      variables~\cite[Eq.     (3.22)]{CJ:16},      where
$H_{x_1,x_2}(\omega)$ is defined in~\cite{CJ:16} and $x=(x_1,x_2)$ and
$y=(y_1,y_2)$ are  as defined in~\eqref{xyscaling}. Catering  for this
alteration immediately gives Theorem~\ref{Airyasymptotics}.
\end{remark} 

As given in~\cite[(4.20)]{CJ:16}, define 
\begin{equation}\label{eq:Ekl}
E_{k,l} = \frac{1}{(2\pi \mathrm{i})^2} \int_{\Gamma_1} \frac{du_1}{u_1}\int_{\Gamma_1} \frac{du_2}{u_2} \frac{1}{\tilde{c}(u_1,u_2) u_1^{k} u_2^{l}}.
\end{equation}
Then, see~\cite[Eq. (4.22)]{CJ:16}, for $x \in \mathtt{W}_{\eps_x}$, $y \in \mathtt{B}_{\eps_y}$,
\begin{equation}\label{KinverseE}
\mathbb{K}^{-1}_{1,1}(x,y)= -\mathrm{i}^{1+h(\eps_x,\eps_y)} ( a^{\eps_y } E_{\mathtt{k}_1,\mathtt{l}_1}+ a^{1-\eps_y } E_{\mathtt{k}_2,\mathtt{l}_2} )
\end{equation}
where
\begin{equation}\label{kl}
\mathtt{k}_1=\frac{x_2-y_2-1}{2}+h(\eps_x,\eps_y), \hspace{1mm}
\mathtt{k}_2=\frac{x_2-y_2+1}{2}-h(\eps_x,\eps_y),\hspace{1mm}
\mathtt{l}_1=\frac{y_1-x_1-1}{2},\hspace{1mm}
\mathtt{l}_2=\frac{y_1-x_1+1}{2}.
\end{equation}
From~\cite[Lemma 4.6 and Lemma 4.7]{CJ:16},  we get the following asymptotic formulas and estimates
\begin{lemma}[\cite{CJ:16}]\label{Ekasymptotics}
Let $A_m,B_m, m\geq 1$ be given and set $b_m=\max(|A_m|,|B_m|)$, and let $a_m=A_m$ if $b_m=|B_m|$, and let $a_m=B_m$ if $b_m=|A_m|$.
\begin{enumerate}
\item Assume that $b_m \to \infty$ as $m \to \infty$ and $|a_m| \leq b_m^{{7}/{12}}$ for large $m$.  Then, there exists a constant $d_1 >0$ so that
\begin{equation}\label{Easymptotics}
E_{B_m+A_m,B_m-A_m}=\frac{(-1)^{a_m+b_m}\mathcal{C}^{2b_m}\left(e^{-\frac{\sqrt{1-2c}}{2c}\frac{a_m^2}{b_m}}\left(1+O\left(b_m^{-{1}/{4}}\right)\right) +O\left(e^{-d_1 b_m^{{1}/{6}}}\right)\right)}{2(1+a^2)(1-2c)^{1/4}\sqrt{2\pi cb_m}}
\end{equation}
as $m\to\infty$.
\item Assume that $b_m >0$, $m\geq 1$.  There exists constants $C,d_1,d_2 >0$ so that
\begin{equation} \label{Ebound}
|E_{B_m+A_m,B_m-A_m} | \leq \frac{C}{\sqrt{b_m}} \mathcal{C}^{2 b_m} \left( e^{-d_1 \frac{a_m^2}{b_m}}+e^{-d_2 b_m} \right)
\end{equation}
for all $m\geq 1$.
\end{enumerate}

\end{lemma}

Motivated by~\eqref{KinverseE} and~\eqref{Easymptotics}, we define
\begin{equation}
A_{m,i}=\frac{\mathtt{k}_i-\mathtt{l}_i}{2} \hspace{5mm}\mbox{and}  \hspace{5mm}B_{m,i}=\frac{\mathtt{k}_i+\mathtt{l}_i}{2}
\end{equation}
for $i \in \{1,2\}$.  It follows from~\eqref{kl} that
\begin{equation}\label{AmBmxy}
\begin{split}
2 A_{m,i}&=\frac{x_1(z')+x_2(z')-(y_1(z)+y_2(z))}{2} -(-1)^{i} h(\eps(z),\eps(z')) \\
2 B_{m,i}&=\frac{x_2(z')-x_1(z')+(y_1(z)-y_2(z))}{2} +(-1)^{i} (1-h(\eps(z),\eps(z'))). \\
\end{split}
\end{equation}
If we have $z \in \mathcal{L}_m(q,k)$, $z' \in \mathcal{L}_m(q',k')$, $t=t(z)$, $t'=t(z')$, $\eps=\eps(z)$ and $\eps'=\eps(z')$, then using~\eqref{tz} and~\eqref{xyz}
\begin{equation}\label{AmBmt}
\begin{split}
2 A_{m,i}&=2(t'-\tau_m(q')) -2(t-\tau_m(q))+2(\eps-\eps')-(-1)^ih(\eps,\eps')\\
2 B_{m,i}&= \beta_m(q,k)-\beta_m(q',k')+\eps+\eps' -1 +(-1)^i (1-h(\eps,\eps'))
\end{split}
\end{equation}
We are now ready for the proof of Proposition~\ref{Asymptotics}.

\begin{proof}[Proof of Proposition~\ref{Asymptotics}]
To prove part (1) in the statement of the proposition, we apply Theorem~\ref{Airyasymptotics}. By comparing~\eqref{tz} and~\eqref{xyscaling}, $y=y(z)$ we see that
\begin{equation}
\alpha_y =\frac{t-\tau_m(q)}{\lambda_1(2m)^{1/3}}, \hspace{5mm}\mbox{and} \hspace{5mm} \beta_y=\beta_q,
\end{equation}
if $z \in \mathcal{L}_m(q,k), t=t(z)$, where we have disregarded integer parts.  Thus, we have
\begin{equation} \label{g1formula}
\alpha_y \beta_y +\frac{2}{3} \beta_y^3 = \frac{t}{\lambda_1 (2m )^{1/3}} \beta_q -\frac{1}{3} \beta_q^3 =\gamma_1(z)
\end{equation}
by~\eqref{gi}.  Using~\eqref{Ktilde},~\eqref{Kmdelta} and~\eqref{g1formula}, we see that part (1) in the statement of the proposition follows from~\eqref{As1}.   Similarly, part (2) in the statement of the proposition follows from~\eqref{As2}.

We now consider part (3) in the statement of the  proposition, that is $q=q'$, $k>k'$. From~\eqref{AmBmt} and the definition of $\beta_m(q,k)$ we see that
\begin{equation}
B_{m,i}=(k-k') \lambda_2 (\log m)^2 +\frac{1}{2} (\eps+\eps' +(-1)^i (1-h(\eps_1,\eps_2)))
\end{equation} 
so $B_{m,i}>0$ if $m$ is sufficiently large.  Also,
\begin{equation}
A_{m,i}=t'-t+\eps-\eps',
\end{equation}
since $t_m(q')=t_m(q)$.  Assume now that $|t'-t| \leq c_2 ((k-k')(\log m)^2)^{7/12}$.  Then, $b_{m,i}=|B_{m,i}|$ and
\begin{equation}
|a_{m,i}| = |A_{m,i}| \leq b_{m,i}^{7/12},
\end{equation}
for large $m$ if $c_2 <1$.  By~\eqref{Easymptotics} 
\begin{equation}
E_{\mathtt{k}_i,\mathtt{l}_i}= \frac{(-1)^{k_i} \mathcal{C}^{b_{m,i}} }{2(1+a^2) (1-2c)^{1/4} \sqrt{2 \pi c b_{m,i}}} \left( e^{-\frac{\sqrt{1-2c}}{2c} \frac{A_{m,i}^2}{B_{m,i}}} (1+O(b_{m,i}^{-1/4})) +O(e^{-d_1 b_{m,i}^{1/6}}) \right).
\end{equation}
Note that
\begin{equation}
-\frac{\sqrt{1-2c}}{2c} \frac{A_{m,i}^2}{B_{m,i}}= -\frac{\lambda_1^2 \sqrt{1-2c}}{2c \lambda_2(k-k')} \left( \frac{t'}{\lambda_1\log m}- \frac{t}{\lambda_1\log m} \right)^2+o(1)
\end{equation}
and that
\begin{equation} \label{g1est}
|\gamma_1(z')-\gamma_1(z)| = \left| \frac{t'-t}{\lambda_1 (2m)^{1/3}} \beta_q \right| \leq C 
\end{equation}
since $|t'-t|\leq C m^{1/3}$.
We can now use~\eqref{Ktilde},~\eqref{Ktilde2},~\eqref{KinverseE} and proceed as in the proof of~\cite[Proposition 3.4]{CJ:16} and this will give part (3)(a) in the proposition.

We turn now to part (3)(b) in the proposition. Consider~\eqref{Ktilde2} and note that
\begin{equation} 
\mathcal{C}^{\gamma_2(z')-\gamma_2(z)+2-2\eps'} =\mathcal{C}^{\frac{1}{2}(x_1(z')-x_2(z')+y_2(z)-y_1(z)+2)} =\mathcal{C}^{-2 B_{m,i} +1-(-1)^i (1-h(\eps,\eps'))}
\end{equation}
by~\eqref{xyg2} and~\eqref{AmBmxy}.  We can now use~\eqref{Ebound} to get
\begin{equation}
|E_{\mathtt{k}_i,\mathtt{l}_i}| \leq \frac{C}{\sqrt{b_{m,i}}} \mathcal{C}^{2 b_{m,i}} (e^{-d_1 \frac{a_{m,i}^2}{b_{m,i}}} + e^{-d_2 b_{m,i}} ).
\end{equation}
If $(c_2 ((k-k')) (\log m)^2)^{7/12} \leq |t'-t| \leq \lambda_2 (k-k')(\log m)^2$, then $b_{m,i}=B_{m,i}$.  The estimate~\eqref{g1est} holds and combining these facts, we obtain the bound in (3)(b) in the statement of the proposition.  

If $|t'-t| \geq \lambda_2 (k-k')(\log m)^2$, then
\begin{equation}
\begin{split}
b_{m,i}& = |A_{m,i}| =|t'-t| +O(1) \\
a_{m,i}& = B_{m,i} = (k-k')\lambda_2 (\log m)^2 +O(1).
\end{split}
\end{equation}
It follows, since $|b_{m,2}-b_{m,1}|$ and $|a_{m,2}-a_{m,1}|$ are bounded, that 
\begin{equation}\label{Kmzerobound}
|\mathcal{K}_{m,0} (z,z')| \leq C \mathcal{C}^{2(b_{m,1}-B_{m,1})} e^{c_1 \frac{-a_{m,1}^2}{b_{m,1}}}.
\end{equation}
If $\lambda_2 (k-k')(\log m)^2 \leq |t'-t| \leq 2 \lambda_2 (k-k')(\log m)^2$, we can use $\mathcal{C} <1$ and $b_{m,1}-B_{m,1} \geq 0$, to get
\begin{equation}
|\mathcal{K}_{m,0} (z,z')| \leq  e^{-c_1 (k-k')(\log m)^2}.
\end{equation}
If $|t'-t|>2 \lambda_2 (k-k')(\log m)^2$, we use $\mathcal{C}<1$ to get
\begin{equation}
|\mathcal{K}_{m,0}(z,z')| \leq C \mathcal{C}^{2(b_{m,1} -B_{m,1})} \leq C\mathcal{C}^{2(k-k')\lambda_2 (\log m)^2} \leq Ce^{-c_1(k-k')(\log m)^2}
\end{equation}
with an appropriate $c_1>0$.   In either case, we have shown (3)(c) in the statement of the proposition.

Consider now the case (4) in the statement of the proposition.  In this case, $B_{m,i}<0$ and we see that the factor 
\begin{equation}
\mathcal{C}^{2(b_{m,i}-B_{m,i})} \label{Gifactor}
\end{equation}
in \eqref{Kmzerobound} will give us the decay we need in order to prove the bound in statement (4) of the proposition. 

Finally, we consider statement (5) in the proposition, that is $q=q'$ and $k=k'$.  Then, we have
\begin{equation}
B_{m,i} = \eps +\eps' -1 +(-1)^i (1-h(\eps,\eps')),
\end{equation}
and
\begin{equation}
A_{m,i} =t' -t +\eps-\eps'.
\end{equation}
Thus,  if $|t' -t|$ is sufficiently large, then $b_{m,i}=|A_{m,i}|$ and $a_{m,i}=B_{m,i}$.  Since $|B_{m,i}| \leq 2$, 
\begin{equation}
2 (b_{m,i}-B_{m,i}) \geq 2 (|t'-t|-2)
\end{equation}
and again the factor in~\eqref{Gifactor} gives the desired bound.

\end{proof}
\bibliographystyle{plain}

\bibliography{ref}

\end{document}